\documentclass[reqno, 11pt]{amsart}
\usepackage{charter,euler}
\usepackage[foot]{amsaddr}
\usepackage{amsfonts}
\usepackage{amsthm,thmtools,thm-restate,enumerate}
\usepackage[margin=1in]{geometry}
\usepackage{algorithm,hyperref}
\usepackage{algpseudocode}
\algrenewcommand{\algorithmiccomment}[1]{\hfill\(\triangleright\) #1} 

\usepackage[mathscr]{euscript}
\usepackage[title]{appendix}

\newtheorem{lemma}{Lemma}[section]
\newtheorem{theorem}[lemma]{Theorem}
\newtheorem{corollary}[lemma]{Corollary}
\newtheorem{proposition}[lemma]{Proposition}

\theoremstyle{definition}
\newtheorem{definition}[lemma]{Definition}
\newtheorem{remark}[lemma]{Remark}

\usepackage{mathtools,tikz, tikz-cd, subcaption}
\usetikzlibrary{arrows,decorations.pathmorphing,backgrounds,calc,math,intersections,angles,patterns}
\tikzstyle{map}=[->,semithick]
\tikzstyle{arc}=[bend left,->,semithick]
\tikzstyle{rinclusion}=[right hook->,semithick]
\tikzstyle{linclusion}=[left hook->,semithick]
\def\s{{\mathcal S}}
\def\Ri{\mathcal R}

\def\R{{\mathbb R}}

\def\B{{\mathbb B}}

\def\K{{\mathcal K}}

\def\L{{\mathcal L}}
\def\eps{{\varepsilon}}

\def\G{{\mathcal{G}}}
\def\sh{{\mathscr{Sh}}}
\def\sc{{\mathscr{SC}}}

\newcommand{\diam}[2][\M]{\mathrm{diam}_{#1}\left(#2\right)}

\newcommand{\sd}[2][]{\mathrm{sd}^{#1}\left(#2\right)}
\newcommand{\norm}[1]{\left\lVert #1 \right\rVert}

\def\rad{\operatorname{rad}}
\def\diam{\operatorname{diam}}
\def\sd{\operatorname{sd}}

\begin{document}

\title[Euclidean Graph Reconstruction]{Vietoris--Rips Shadow for Euclidean Graph Reconstruction}

\author{Rafal Komendarczyk}
\address{Mathematics Department, Tulane University, USA}
\email{rako@tulane.edu}

\author{Sushovan Majhi}
\address{Data Science Program, George Washington University, USA}
\email{s.majhi@gwu.edu}

\author{Atish Mitra}
\address{Mathematics Department, Montana Technological University, USA}
\email{amitra@mtech.edu}







\begin{abstract}
The shadow of an abstract simplicial complex $\mathcal{K}$ with vertices in $\mathbb{R}^N$ is a subset of $\mathbb{R}^N$ defined as the union of the convex hulls of simplices of $\mathcal{K}$.
The Vietoris--Rips complex of a metric space $(\mathcal S,d)$ at scale $\beta$ is an abstract simplicial complex whose each $k$-simplex corresponds to $(k+1)$ points of $\mathcal S$ within diameter $\beta$.
In case $\mathcal{S}\subset\mathbb R^2$ and $d(a,b)=\|a-b\|$ the standard Euclidean metric, the natural shadow projection of the Vietoris--Rips complex is already proved by Chambers et al. to induce isomorphisms on $\pi_0$ and $\pi_1$.
We extend the result beyond the standard Euclidean distance on $\mathcal{S}\subset\mathbb{R}^N$ to a family of path-based metrics, $d^\varepsilon_{\mathcal S}$.
From the pairwise Euclidean distances of points in $\mathcal S$, we introduce a family (parametrized by $\varepsilon$) of path-based Vietoris--Rips complexes $\mathcal{R}^\varepsilon_\beta(\mathcal S)$ for a scale $\beta>0$. 
If $\mathcal{S}\subset\mathbb{R}^2$ is Hausdorff-close to a planar Euclidean graph $\mathcal G$, we provide quantitative bounds on scales $\beta,\varepsilon$ for the shadow projection map of the Vietoris--Rips complex of $(\mathcal{S},d^\varepsilon_\mathcal{S})$ at scale $\beta$ to induce $\pi_1$-isomorphism.
This paper first studies the homotopy-type recovery of $\mathcal G\subset\mathbb R^N$ using the abstract Vietoris--Rips complex of a Hausdorff-close sample $\mathcal{S}$ under the $d^\varepsilon_\mathcal{S}$ metric. 
Then, our result on the $\pi_1$-isomorphism induced by the shadow projection lends itself to providing also a geometrically close embedding for the reconstruction.
Based on the length of the shortest loop and large-scale distortion of the embedding of $\mathcal G$, we quantify the choice of a suitable sample density $\varepsilon$ and a scale $\beta$ at which the shadow of $\mathcal{R}^\varepsilon_\beta(\mathcal{S})$ is homotopy-equivalent and Hausdorff-close to $\mathcal G$.
\end{abstract}



\maketitle




\section{Introduction}
Given a metric space $(\s,d_\mathcal{\s})$ and scale $\beta>0$, the \emph{Vietoris--Rips complex}, denoted $\Ri_\beta(\s)$, is defined as an abstract simplicial complex having a $k$-simplex for every subset of $\s$ with cardinality $k+1$ and diameter less than $\beta$.
In the last decade, Vietoris--Rips complexes have gained considerable attention in the topological data analysis community due to their relatively straightforward computational schemes regardless of the dimension of the data, compared to some of the alternatives such as \v{C}ech and $\alpha$-complexes.
The theoretical understanding of the topology of Vietoris--Rips complexes at different scales is generally extremely elusive.
Nonetheless, the far and wide use of Vietoris--Rips complexes---particularly in the field of shape reconstruction---can be attributed to their ability to \emph{reconstruct} the topology of a hidden shape $\mathcal X$ even when constructed on a noisy sample $\s$ ``around'' $\mathcal X$; \cite{hausmann1995vietoris,attali2011vietoris,kim2020homotopy,majhi2023vietoris,majhi2024demystifying,MajhiStability}. 

There are many real-world applications where the unknown shape $\mathcal{X}$ and the sample $\s$ are hosted in a Euclidean space $\R^N$ within a small Hausdorff proximity (Definition~\ref{def:dH}).
In case $\mathcal{X}$ belongs to a nice enough class of shapes, the Vietoris--Rips complex of $\s$ (possibly under a non-Euclidean metric) has been shown to successfully reconstruct $\mathcal{X}$ up to homotopy-type; pivotal developments include \cite{attali2011vietoris,kim2020homotopy} for compact subsets of positive reach, \cite{majhi2023vietoris} for Euclidean submanifolds, \cite{majhi2023vietoris} for Euclidean graphs, \cite{MajhiStability} for more general geodesic subspaces of curvature bounded above.

This paper is devoted to Euclidean graph reconstruction---both \emph{topologically} and \emph{geometrically}.
Graph structures are ubiquitous in real-world applications.
In practice, Euclidean data or \emph{sample} $\s\subset\R^N$ sometimes approximate an (unknown) graph $\G$ that is realized as a subset of the same Euclidean space $\R^N$ with a controlled Hausdorff distance $d_H(\s,\G)$.  
For compact sets of positive reach or $\mu$-reach, powerful reconstruction results are available~\cite{attali2011vietoris,chazal2006sampling,kim2020homotopy}. However, Euclidean graphs---even smooth ones---generically have vanishing $\mu$-reach at their vertices, placing them outside the scope of those results. 
Our conditions, based on the length of the shortest loop and the large-scale
distortion, provide an alternative framework better suited to the graph setting and more convenient in contexts where $\mu$-reach bounds are unavailable.

\paragraph{Homotopy Reconstruction} Our study of the homotopy reconstruction of Euclidean graphs via Vietoris--Rips complexes of the sample is inspired by the recent developments in the reconstruction of 
graphs~\cite{majhi2023vietoris} and general geodesic subspaces~\cite{MajhiStability}, using a non-Euclidean, path-based metric for the output Vietoris--Rips complexes. 
The sample $\s$ comes equipped with the Euclidean distance between pairs of points. 
Even when such a sample exhibits a sufficiently small Hausdorff--closeness to $\G$, the \emph{Euclidean}  Vietoris--Rips complex generally fails to be homotopy equivalent to the underlying graph. 
Near the vertices of $\G$, the presence of small redundant cycles in the Euclidean Vietoris--Rips complex of $\s$ is often unavoidable; see Figure~\ref{fig:5-prong-shadow-MA}.

For this reason, the Euclidean metric on $\s$ is not deemed an appropriate metric for building the Vietoris--Rips complexes on $\s$ to obtain a topologically faithful reconstruction of the unknown graph. 
Instead, the authors of~\cite{fasy2022reconstruction,majhi2023vietoris,MajhiStability} considered the Vietoris--Rips complexes of the sample under a family of path-based metrics $(\s,d^\eps_\s)$ (defined in Definition~\ref{def:eps-path}) in their reconstruction schemes. 
Under this metric, for sufficiently small scale $\beta$, we show that the Vietoris--Rips complex $\Ri^\eps_\beta(\s)$ is homotopy equivalent to $\G$.

Our homotopy reconstruction extends and improves the above-mentioned works in the following directions. 
Although the authors of~\cite{fasy2022reconstruction,majhi2023vietoris} considered embedded graph reconstruction in a similar setting, a noteworthy limitation was the use of the global distortion of $\G$, which is known to become infinite in the presence of the cusp-like structures in $\G$.
We successfully mitigate the caveat by putting forward the large-scale distortion (Definition~\ref{def:lc-dist}) as a more robust, alternative sampling parameter in our reconstruction scheme. 
In addition, our proof techniques are much simpler than~\cite{majhi2023vietoris}, with reconstruction guarantees under much weaker sampling conditions.
We also mention that the large-scale distortion was introduced in~\cite{MajhiStability} for the reconstruction of spaces more general than Euclidean graphs.
However, we present a more direct proof, which avoids using their two main ingredients---Hausmann's theorem~\cite{hausmann1995vietoris} and Jung's theorem~\cite{lang1997jung}---resulting in a much weaker sampling condition in the special case of graphs.

Based on the length $\ell(\G)$ of the shortest loop and large-scale distortion $\delta_\beta^\eps(\G)$ of the embedding of $\G$, we show how to choose a suitable density parameter $\varepsilon$ and a scale $\beta$ such that $\mathcal{R}^\varepsilon_\beta(\s)$ is homotopy-equivalent to $\G$. 
\begin{restatable*}[Homotopy Reconstruction]{theorem}{homeq}\label{thm:hom-eq}
Let $\G \subset \mathbb{R}^N$ be a compact, connected metric graph.  
Fix any $\xi\in\left(0,\frac{1}{4}\right)$.
For any positive $\beta<\frac{\ell(\G)}{4}$, choose\footnote{It is possible to choose such small $\eps$, as $\delta^\eps_\beta(\G)\to1$ as $\eps\to0$.} 
a positive $\eps\leq\frac{\beta}{3}$ such that $\delta^{\eps}_{\beta}(\G)\leq\frac{1+2\xi}{1+\xi}$. 
If $\s\subset \R^N$ is such that $d_H(\G,\s)<\tfrac{1}{2}\xi\eps$, then we have a homotopy equivalence $\Ri^\eps_\beta(\s)\simeq \G$.
\end{restatable*}

\paragraph{Geometric Reconstruction} Topologically faithful reconstructions are only useful to estimate the homological features---such as the Betti numbers, Euler characteristic, etc---of the hidden shape $\mathcal X$.
A more challenging yet more useful paradigm is \emph{geometric reconstruction}: 
to output a subset $\widetilde{\mathcal X}$ in the same host Euclidean space $\R^N$ computed from $\s$ such that $\widetilde{\mathcal X}$ is not only homotopy equivalent but also Hausdorff-close to $\mathcal X$.


Despite aiding in homotopy equivalent reconstruction, as an abstract simplicial complex, Vietoris--Rips complexes fail to provide an embedding in the same host Euclidean space.
For a geometric reconstruction of Euclidean shapes, it's most natural to consider the shadow of the Vietoris--Rips complexes.
The \emph{shadow} of an abstract simplicial complex $\K$ with vertices in $\R^N$ is a subset of $\R^N$ defined as the union of the convex hulls of simplices of $\K$; see Definition~\ref{def:shadow} for more details.

The shadow of a general simplicial complex with Euclidean vertices is notorious for being topologically unfaithful.
However, when considering the Vietoris--Rips complex of a finite points in $\R^2$ under the Euclidean metric, the shadow project map has been shown in \cite{Chambers2010} to induce isomorphisms on both $\pi_0$ and $\pi_1$.
Furthermore, the authors show that the projection map fails to induce surjection on $\pi_1$ for any $N\geq4$ and fails to induce an injective homomorphism on $\pi_k$ for any $N\geq2$ and $k\geq2$.
The curious case of $N=3$ was later partially resolved in \cite{adamaszek_homotopy_2017} by proving that the shadow projection induces a surjection on $\pi_1$.

In this paper, we consider the Vietoris--Rips complexes of a sample $\s\subset\R^2$, constructed under a (possibly non-Euclidean) family (parametrized by $\eps$) of path-based metrics $d^\eps_\s$ on $\s$.
The phenomenal utility of such path-based metrics has recently been demonstrated by the authors of \cite{fasy2022reconstruction,majhi2023vietoris,MajhiStability} in the context of shape reconstruction beyond smooth submanifolds.
If $\s$ is Hausdorff-close to a Euclidean graph $\G$, we provide quantitative bounds on scales $\beta,\eps$ for the shadow projection map of the Vietoris--Rips of $(\s,d^\eps_\s)$ at scale $\beta$ to induce $\pi_1$-isomorphism.
This leads to the following pragmatic geometric reconstruction scheme using the quantity $\Theta$ (defined in \ref{eq:theta}) and the shadow radius $\Delta(\G)$ of $\G$ as introduced in Definition~\ref{def:shadow}.

\begin{restatable*}[Geometric Reconstruction]{theorem}{geom}\label{thm:geom}
Let $\G \subset \mathbb{R}^2$ a graph having properties (A1--A4) as described in Section~\ref{sec:assumptions}.
Fix any $\xi\in\left(0,\frac{1-\Theta}{6}\right)$.
For any positive $\beta<\min\left\{\Delta(\G),\frac{\ell(\G)}{18}\right\}$, choose\footnote{It is possible to choose such small $\eps$, as $\delta^\eps_\beta(\G)\to1$ as $\eps\to0$.} 
a positive $\eps\leq\frac{(1-\Theta)(1-\Theta-6\xi)}{12}\beta$ such that $\delta^{\eps}_{\beta}(\G)\leq\frac{1+2\xi}{1+\xi}$. 
If $\s\subset \R^2$ is such that $d_H(\G,\s)<\tfrac{1}{2}\xi\eps$, then the shadow $\sh(\Ri_\beta^\eps(\s))$ is homotopy equivalent to $\G$. Moreover, $d_H(\sh(\Ri_\beta^\eps(\s)),\G)<\left(\beta+\frac{1}{2}\xi\eps\right)$. 
\end{restatable*}

\subsection{Organization of the Paper}
The paper is organized in the following manner.
We present the basic definitions and results from topology and metric geometry in~Section~\ref{sec:preli}.
Section~\ref{sec:top} presents our result on the homotopy reconstruction of $\G$.
In Section~\ref{sec:shadow}, we define the shadow of a general simplicial complex and study conditions to achieve surjectivity 
of the natural shadow projection on the fundamental group.
Finally, Section~\ref{sec:geom} defines our novel sampling parameter shadow radius in order to provide a Hausdorff-close, homotopy-equivalent geometric reconstruction of $\G$.

\section{Preliminaries}\label{sec:preli}
This section provides important notations and definitions along with basic results from algebraic topology and metric geometry used throughout the paper. 

\subsection{Basic Definitions}
Let $\K$ be an abstract simplicial complex.
We denote by $|\K|$ the geometric realization of $\K$.
\begin{definition}[\v{C}ech Complex]\label{def:cech}
Let $(\mathcal{X},d_\mathcal{X})$ be a metric space and $\mathcal{A}\subset\mathcal{X}$ be a subset.
For scale $\beta >0$, the \emph{\v{C}ech complex}, denoted $\check{\mathcal{C}}_\beta(\mathcal A)$, is defined as the abstract simplicial complex whose simplices are the finite subsets of points of $\mathcal A$ such that the (open) $\beta$-balls about those points have a common intersection in $\mathcal{X}$.
\end{definition}

\begin{definition}[Hausdorff Distance]\label{def:dH}
Let $(\mathcal X,d_\mathcal{X})$ be a metric space. 
Let $\mathcal A$ and $\mathcal B$ be compact, non-empty subsets. 
The \emph{Hausdorff distance} between them, denoted $d^\mathcal{X}_H(\mathcal A, \mathcal B)$, is defined as
\[
d^\mathcal{X}_H(\mathcal A, \mathcal B) \coloneqq 
\max\left\{\sup_{a\in\mathcal A}\inf_{b\in \mathcal B}d_\mathcal{X}(a,b),\sup_{b\in \mathcal B}\inf_{a\in\mathcal A}d_\mathcal{X}(a,b)\right\}.
\]
In case $\mathcal X\subset\mathbb{R}^N$ and $\mathcal A,\mathcal B,\mathcal{X}$ are all equipped with the Euclidean metric, we simply write $d_H(\mathcal A,\mathcal B)$.
\end{definition}

\subsection{Embedded Metric Graphs}\label{sec:metric-graphs}
Let $\G\subset\R^N$ be an embedded metric graph, i.e., a metric graph \cite{majhi2023vietoris} such that the intrinsic metric coincides with the length metric induced from the Euclidean subspace metric. 
$\G$ comes equipped with the standard Euclidean metric $\|\bullet-\bullet\|$.
Using the notion of the length $L(\gamma)$ of a continuous path $\gamma$ in $\R^N$, we can define the intrinsic or shortest path metric $d_\G$ on $\G$ as follows:
For any two points $a,b\in \G$, it is the infimum of the lengths of paths $[0,1]\to\G$ joining $a,b$. 
In this metric, the diameter of a bounded subset $\mathcal A\subset\G$ is denoted by $\diam_\G(\mathcal A)$.
If $\G$ is path-connected, $d_\G$ defines a metric on $\G$, called the \emph{length} metric. 
Using the above definition, we are also allowing $\G$ to have single-edge cycles and multiple edges between a pair of vertices.
We denote by $\ell(\G)$ the length of the smallest simple cycle (systole) of $\G$. In case $\G$ is contractible, we use the convention that $\ell(\G)=\infty$. 
We also denote by $\mathscr{V}(\G)$ and $\mathscr{E}(\G)$ the vertices and edges of $\G$, respectively.
\begin{remark}
We assume throughout the paper that $\G$ is compact with $\mathscr{E}(\G)<\infty$.
As a result, $\ell(\G)>0$.    
\end{remark}

\subsection{Hausmann-Type Theorem for Metric Graphs}\label{subsec:circumcenter} We begin with the definition of circumradius and circumcenter in $(\G,d_\G)$, with respect to the length metric. 
For a bounded subset $\mathcal A\subset \G$, we define its \emph{circumradius} to be the radius of the smallest (closed) $d_\G$-metric ball enclosing $\mathcal A$. More formally,
\[
\rad(\mathcal A)\coloneqq\inf_{g\in \G}\ \sup_{a\in\mathcal A}d_\G(a, g).
\]
A point $g\in \G$ satisfying $\max_{a\in\mathcal A}d_\G(a ,g)=\rad(\mathcal A)$ is called a \emph{circumcenter} of $\mathcal A$, and is denoted by $c(\mathcal A)$. 
The reader must be warned that these terms may not always match their usual meaning in plane geometry.

In light of the above definition, we present the following important proposition.
\begin{proposition}[Circumcenter Existence]\label{prop:center} Let
$\mathcal A\subset(\G,d_\G)$ be a (non-empty) compact subset with $\diam_\G(\mathcal A)<\ell(\G)/3$. 
Then, the circumcenter of $\mathcal A$ exists uniquely. 
Moreover, the circumradius is $\frac{1}{2}\diam_\G(\mathcal A)$.
\end{proposition}
\begin{proof}
Since $\mathcal A\subset\G$ is compact, there exist $a,b\in\mathcal A$ such that $\diam_\G(\mathcal A)=d_\G(a,b)$.
Let $\gamma\subset\G$ be the unique geodesic joining $a,b$ and $c$ their midpoint.
Since $d_\G(a,b)<\ell(\G)/3$, the midpoint of $\gamma$ is the unique circumcenter of $\mathcal A$, and that the ball around $c$ of radius $\frac{1}{2}d_\G(a,b)$ contains $\mathcal A$ entirely. 

Indeed, if that were not the case, there would be a point $a'\in\mathcal A$ outside of the ball with distances to both $a$ and $b$ at most $\diam_\G(\mathcal A)<\ell(\G)/3$.
Since $\G$ is a graph, this is possible only when concatenating $\gamma$ with the geodesics joining $c,a$ and $c,b$ forms a cycle in $\G$.
Since the length of $\gamma$ is also less than $\ell(\G)/3$, this would contradict the very definition of $\ell(\G)$.
\end{proof}

A proof of the following crucial proposition is presented in the appendix.
\begin{proposition}[Circumcenter Distances]\label{prop:radius} If $\mathcal A$ is a (non-empty) compact subset of $(\G, d_\G)$ with $\diam_\G(\mathcal A)<\ell(\G)/3$, then for any non-empty subset $\mathcal A'\subset \mathcal A$, we have
\[d_\G\left(c(\mathcal A'),c(\mathcal A)\right) \leq\tfrac{1}{2}\diam_\G(\mathcal A).\]
\end{proposition}

As a consequence of Proposition~\ref{prop:center}, we get the following improved result for graphs in the style of Hausmann's theorem~\cite{hausmann1995vietoris}.
Throughout the paper, $\check{\mathcal C}^L_\beta(\G)$ and $\Ri^L_\beta(\G)$ are used to denote the \v{C}ech and Vietoris--Rips complex of $\G$ with respect to the length metric $(\G, d_\G)$, respectively.
Since $\G$ is infinite, they are infinite complexes. 

\begin{theorem}[Hausmann-Type Theorem for Metric Graphs]\label{thm:hausmann-graph}
For any compact metric graph $(\G,d_\G)$ and $0<\beta<\ell(\G)/3$, we have $\Ri^L_\beta(\G)\simeq\G$.

Furthermore, if $0<\beta\leq\beta'<\ell(\G)/3$, then the natural inclusion $\Ri^L_\beta(\G)\xhookrightarrow{\iota}\Ri^L_{\beta'}(\G)$ is a homotopy equivalence. 
\end{theorem}
\begin{proof}
For  any $\beta>0$, we have the natural inclusion $\check{\mathcal{C}}^L_{\beta/2}(\G)\hookrightarrow\mathcal{R}^L_{\beta}(\G)$, but in general $\check{\mathcal{C}}^L_{\beta/2}(\G)$ is a proper subcomplex of $\mathcal{R}^L_{\beta}(G)$. 
However, if $0<\beta <\ell(\G)/3$, we can apply proposition \ref{prop:center} to also get $\Ri^L_{\beta}(\G)\hookrightarrow\check{\mathcal{C}}^L_{\beta/2}(\G)$.
So, $\Ri^L_{\beta}(\G)=\check{\mathcal{C}}^L_{\beta/2}(\G)$.
On the other hand, $\check{\mathcal{C}}^L_{\beta/2}(\G)\simeq\G$ due to the Nerve lemma and the fact that the intrinsic balls of radius less than $\ell(\G)/4$ form a good cover.
\end{proof}

\subsection{The $\eps$-path Metric}
Let $\mathcal A\subset\R^N$ be any subset, equipped with the standard Euclidean metric.
The notion of an $\eps$--path was introduced in~\cite{majhi2023vietoris,fasy2022reconstruction} as a family of piecewise Euclidean segments formed by ``hopping over'' the points in $\mathcal A$. 
The family gives rise to yet another alternative (possibly non-Euclidean) metric on $\mathcal A$. For a positive number $\eps$, we first introduce the notion of an $\eps$-path and then the $d^\eps_{\mathcal A}$-metric.

Let $\mathcal A\subset\R^N$ be non-empty and $\eps>0$ a number. 
For $p,q\in\mathcal A$, an \emph{$\eps$--path} of $\mathcal A$ from $p$ to $q$ is a finite sequence $\mathcal{P}=\{a_i\}_{i=0}^{k+1}\subseteq\mathcal A$ such that $a_0=p$, $a_{k+1}=q$, and $\|a_i-a_{i+1}\|<\eps$ for all $i=0,1,\ldots,k$. 
We define the length of the path by $L(\mathcal{P})\coloneqq\sum_{i=0}^k\|a_i-a_{i+1}\|$ and the set of all $\eps$--paths of $\mathcal A$ joining any pair $p,q\in\mathcal A$ by $\mathscr{P}_{\mathcal A}^\eps(p,q)$.
\begin{definition}[The $d^\eps_{\mathcal A}$-metric]\label{def:eps-path}
Let $\mathcal A\subset\R^N$ be non-empty and $\eps>0$ a number. 
The $\eps$--path metric on $\mathcal A$, denoted $d^\eps_{\mathcal A}$, between any $p,q\in\mathcal A$ is defined by 
\begin{equation*}\label{eq:d^eps}
d^\eps_{\mathcal A}(p,q)\coloneqq\inf\left\{L(\mathcal{P})\mid \mathcal{P}\in \mathscr{P}^\eps_{\mathcal A}(p,q)\right\}.
\end{equation*}
\end{definition}
Note that $d^\eps_\mathcal A$ may not be a metric in general. When $d^\varepsilon_\mathcal A$ is finite, however, $(\mathcal A, d^\eps_\mathcal A)$ becomes a metric space.
Under the sampling condition $d_H(\mathcal{G},\mathcal{S})<\tfrac{1}{2}\xi\varepsilon$ with $\mathcal{G}$ path-connected, every pair of points in $\mathcal{S}$ is connected by an $\varepsilon$-path (cf.\ Proposition~\ref{prop:path-connected}), so $d^\varepsilon_{\mathcal{S}}$ is finite and hence a metric throughout the paper.
We use $\diam^\eps_\s(\mathcal A)$ to denote the diameter of a subset $\mathcal A\subset\s\subset\R^N$ in the $d^\eps_\s$ metric.

\begin{proposition}[Comparison of Path Metrics \cite{majhi2023vietoris}]\label{prop:d^esp-d^L-estimate} Let $\G \subset \mathbb{R}^N$ be a graph and $\s \subset \mathbb{R}^N$ such that $d_H(\s, \G) < \frac{1} {2}\xi\varepsilon$ for some $\xi \in (0,1)$ and $\eps > 0$. 
For any $a, b \in \G$ and corresponding $A, B \in \s$ with $||a-A||, ||b-B||<\frac{1}{2}\xi\eps$, we have
\begin{equation}\label{eq:d^eps-d^L-si-eps}
\norm{A-B}\leq d^\eps_\s(A,B)\leq \dfrac{d_\G(a,b) + \xi\varepsilon}{1-\xi}.
\end{equation}
\end{proposition}




\subsection{The Large-Scale Distortion}
Following the definition of~\cite{MajhiStability}, we now define
the large-scale distortion of $\G\subset\R^N$, denoted $\delta^\eps_R(\G)$, parametrized by $\eps>0$ and $R>0$.
\begin{definition}[Large-scale Distortion]\label{def:lc-dist}
For $\eps>0$ and $R>0$, the \emph{large-scale distortion} or \emph{$(\eps,R)$-distortion} $\mathcal \G\subset\R^N$ is defined by
\begin{equation}\label{eq:delta-eps-R}
\delta^\eps_R(\G)\coloneqq \sup_{d_\G(a,b)\geq R} \frac{d_\G(a,b)}{d_{\G^\eps}(a,b)},
\end{equation}
where $\G^\eps$ and $d_{\G^\eps}(a,b)$ denote, respectively, the Euclidean thickening of $\G$ and the length metric thereof induced from the Euclidean metric; see Figure~\ref{fig:large-scale-dist} for illustration. 
\end{definition}

\begin{figure}[htb]
\centering
\includegraphics[width=0.4\linewidth]{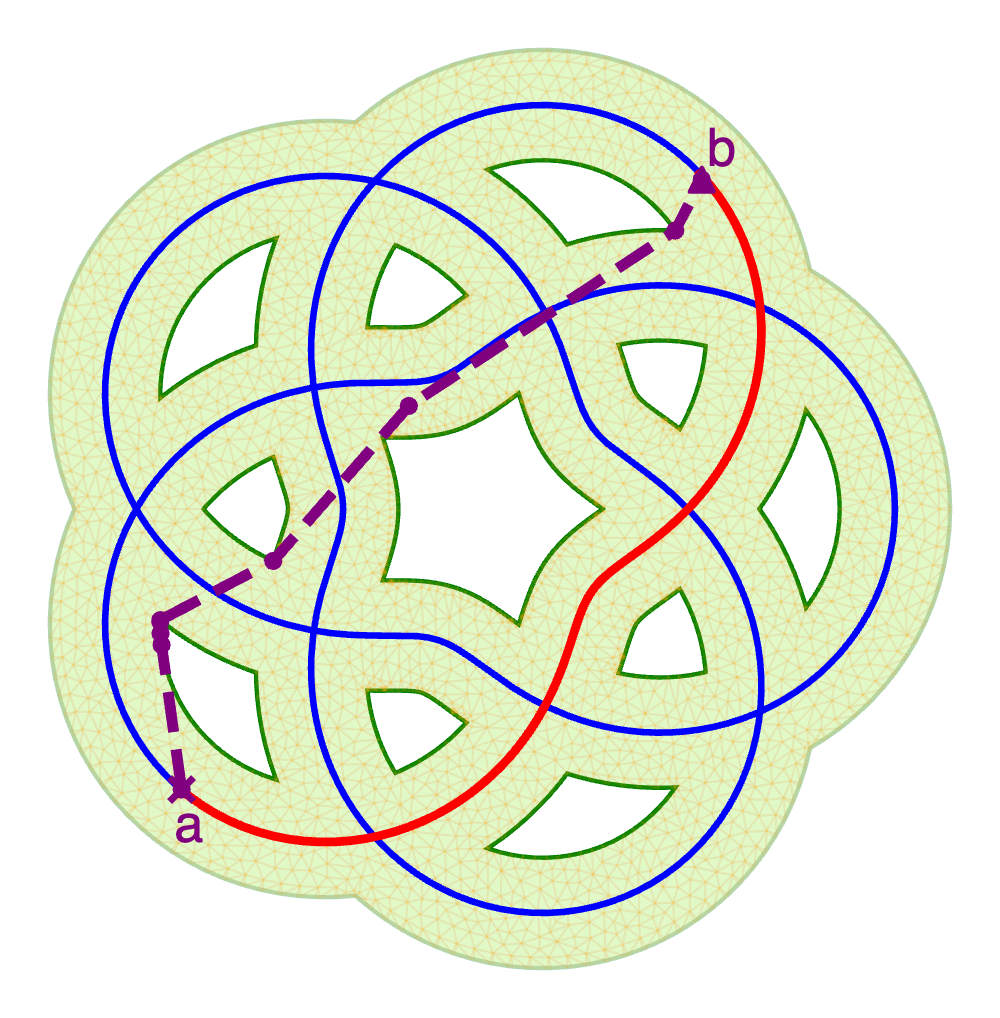}
\caption{A pair of geodesics between points $a,b\in \G$ the red is intrinsic in $\G$ (blue), with length $d_\G(a,b)\geq R$, the purple is intrinsic in $\G^\varepsilon$ (light green tube) with length $d_{\G^\varepsilon}(a,b)$. The large-scale distortion in~\eqref{eq:delta-eps-R} maximizes the ratio of $d_\G(a,b)$ and $d_{\G^\varepsilon}(a,b)$ over all pairs of points in $\G$ with $d_\G(a,b)\geq R$.}
\label{fig:large-scale-dist}
\end{figure}

For more details on the length metric on the thickening and the large-scale distortion, see~\cite{MajhiStability}.
For the purpose of our paper, we only note here that
\begin{enumerate}[(a)]
    \item $\delta^\eps_R(\G)$ is a non-decreasing function of $\eps$ and that $\delta^\eps_R(\G)$ is a non-increasing function of $R$, and
    \item for any $R>0$, the large-scale distortion $\delta^\eps_R(\G)\to1$ as $\eps\to0$.
\end{enumerate}
To conclude the preliminaries, we state the following relevant fact, which has been proved in~\cite{MajhiStability} for general geodesic spaces.
For completeness, a proof is laid out in the appendix.
\begin{proposition}\label{prop:delta_path}
Let $\G\subset\R^N$ and $\s\subset\R^N$ compact with $d_H(\G, \s)<\alpha$. 
Let $\eps>2\alpha$ and $R$ be positive numbers. 
For any $a,b\in \G$ with $d_\G(a,b)\geq R$ and $A,B\in\s$ such that $\max\{\|a-A\|,\|b-B\|\}<\alpha$, we must have
\[
d_\G(a, b)\leq\delta^\eps_R(\G)\left[d^\eps_\s(A, B)+2\alpha\right].
\]
\end{proposition}

\section{Homotopy Reconstruction via Vietoris--Rips Complexes}\label{sec:top}
We first present our homotopy reconstruction result for graphs using Vietoris--Rips complexes.
Throughout this section, we always assume that $\G\subset\R^N$ is a compact, connected Euclidean graph and $\s\subset\R^N$ a compact sample with a small Hausdorff distance to $\G$.

When endowed with the length metric $d_\G$, we denote the Vietoris--Rips complex of $\G$ by $\Ri_\beta^L(\G)$. 
And, the Vietoris--Rips complex of $(\s,d^\eps_\s)$ is denoted by $\Ri_\beta^\eps(\s)$.

Now, we prove the main homotopy reconstruction result. 
We remark that the technique of the proof follows the steps of the previous works \cite{MajhiStability,majhi2023vietoris,majhi2024demystifying}.

\homeq
\begin{proof} [Proof of Theorem \ref{thm:hom-eq}]
Since $d_H(\G,\s)<\tfrac{1}{2}\xi\eps$, we can define a function $\Phi\colon\G\to\s$ such that $\|g-\Phi(g)\|<\tfrac{1}{2}\xi\eps$ for any $g\in\G$. 
Similarly, we can define a function $\Psi\colon\s\to\G$ such that $\|\Psi(s)-s\|< \tfrac{1}{2}\xi\eps$ for any $s\in \s$. 

We can (linearly) extend these maps to simplicial maps to get the following chain of simplicial maps, which we justify shortly:
\begin{align}\label{eq:rips_chain}
\Ri^L_{(1-\xi)\beta - \xi\eps}(\G) 
\xrightarrow{\quad\Phi\quad}\Ri^{\eps}_{\beta}(\s) \xrightarrow{\quad\Psi\quad}\Ri^L_{4\beta/3}(\G).
\end{align}
To justify the map $\Phi$, take $g_1,g_2\in\G$ with $d_\G(g_1,g_2) < (1-\xi)\beta - \xi\eps$. 
The second inequality in \eqref{eq:d^eps-d^L-si-eps} then implies 
\begin{equation}\label{eq:phi_beta}
d_\s^\eps(\Phi(g_1),\Phi(g_2))\leq \dfrac{d_\G(g_1,g_2) + \xi\varepsilon}{1-\xi} < \frac{[(1-\xi)\beta - \xi\eps]+\xi\eps}{1-\xi}= \beta.
\end{equation}
Consequently, for any simplex $\sigma\in\Ri^L_{(1-\xi)\beta - \xi\eps}(\G)$,
we must have that $\Phi(\sigma)\in\Ri^{\eps}_{\beta}(\s)$.

We argue similarly to justify the map $\Psi$. 
Take $s_1, s_2\in \s$ with $d^\eps_\s(s_1, s_2)<\beta$. 
In case we already have $d_\G(\Psi(s_1),\Psi(s_2)) < 4\beta/3$, there is nothing to show. 
Otherwise, if $d_\G((\Psi(s_1),\Psi(s_2))\geq4\beta/3>\beta$, then we get from Proposition~\ref{prop:delta_path}:
\begin{align}
d_\G(\Psi(s_1), \Psi(s_2)) 
& \leq \delta^{\eps}_{\beta}(\G) [d^\eps_\s(s_1, s_2)+\xi\eps]\nonumber \\
& \leq \frac{1+2\xi}{1+\xi}(d^\eps_\s(s_1, s_2)+\xi\eps),\text{ due to the choice of }\eps\nonumber \\
& < \frac{1+2\xi}{1+\xi}(\beta+\xi\eps)\label{eq:diam-estimate-1}\\
& \leq \frac{(1+2\xi)(3+\xi)}{3(1+\xi)}\beta,\text{ since }\eps\leq\beta/3\nonumber \\ 
&<\frac{4}{3}\beta,\text{ since }0<\xi<1/4\label{eq:diam-estimate-2}.
\end{align}
For any $s_1, s_2\in \s$ with $d^\eps_\s(s_1, s_2)<\beta$, this shows that $d_\G(\Psi(s_1),\Psi(s_2)) <4\beta/3$, This implies $\Psi$ is a well-defined simplicial map as claimed.

\noindent{\bf Contiguity.}\ The composition $(\Psi\circ\Phi)$ is contiguous (see \cite[Ch. 12]{munkres2018elements} for a definition) to the natural inclusion:
\[
\iota:\Ri^L_{(1-\xi)\beta - \xi\eps}(\G)\hookrightarrow\Ri^L_{4\beta/3}(\G).
\]
Indeed, for any $m$-simplex $\sigma_m\coloneqq[g_0, \ldots, g_m]\in \Ri^L_{(1-\xi)\beta - \xi\eps}(\G)$, it suffices to show 
\[
\operatorname{diam}_\G(\iota(\sigma) \cup \Psi \circ \Phi(\sigma)) < 4\beta/3.
\]
For any $0\leq i,j\leq m$, it suffices to establish that $d_\G(g_i, \Psi\circ\Phi(g_j))<4\beta/3$. 
Without any loss of generality, we assume that  $d_\G(g_i,\Psi\circ\Phi(g_j))\geq4\beta/3>\beta$. 
Then, we have from Proposition~\ref{prop:delta_path}
\begin{equation*}\label{eq:contiguity-ineq}
\begin{split}
d_\G(g_i, \Psi\circ\Phi(g_j)) 
&\leq \delta^{\eps}_{\beta}(\G)[d^\eps_\s(\Phi(g_i), \Phi(g_j))+\xi\eps]
\leq \frac{1+2\xi}{1+\xi}\cdot[d^\eps_\s(\Phi(g_i), \Phi(g_j))+\xi\eps]\\
&\leq \frac{1+2\xi}{1+\xi}\left[\frac{d_\G(g_i,g_j)+\xi\eps}{1-\xi}+\xi\eps\right]\\
&<\frac{1+2\xi}{1+\xi}\left[\frac{(1-\xi)\beta-\xi\eps+\xi\eps}{1-\xi}+\xi\eps\right],\text{ as }d_\G(g_i,g_j) < (1-\xi)\beta - \xi\eps\\
&<\frac{1+2\xi}{1+\xi}\left[\beta\xi+\xi\eps\right]\\
& < \frac{4}{3}\beta,\text{ from~(\ref{eq:diam-estimate-2})}.\\
\end{split}
\end{equation*}

\noindent{\bf Injectivity.}\  
Since we have assumed $4\beta/3 < \ell(\G)/3$, Theorem~\ref{thm:hausmann-graph} implies that the natural inclusion $\iota: \Ri^L_{(1-\xi)\beta - \xi\eps}(\G) 
\hookrightarrow\Ri^L_{4\beta/3}(\G)$ is a homotopy equivalence. 
Since the composition $(\Psi\circ\Phi)$ is contiguous to $\iota$, the map $\Phi$ must induce an injective homomorphism on all the homotopy groups.

\noindent{\bf Surjectivity.}\ Since the geometric realizations $|\Ri^L_{(1-\xi)\beta-\xi\eps}(\G)|$ and $|\Ri^\eps_\beta(\s)|$ are path-connected (due to Proposition~\ref{prop:path-connected} and $\beta\geq\eps$), the result holds for $\pi_0$, i.e., for $m=0$.

For $m\geq1$, let us take an abstract simplicial complex $\K$ such that $|\K|$ is a triangulation of the $m$--dimensional sphere $\mathbb S^m$. 
In order to show surjectivity of $\pi_m(\Phi)$ on $\pi_m(\Ri^L_{(1-\xi)\beta-\xi\eps}(\G))$, we start with a simplicial map  $g\colon\K\to\Ri^\eps_\beta(\s)$, which, by Simplicial Approximation Theorem \cite{munkres2018elements}, approximates any given continuous map $\mathbb S^m=|\K|\longrightarrow |\Ri^\eps_{\beta}(\s)|$ up to homotopy. 
We claim that there exists a simplicial map
$\widetilde{g}\colon\sd\K\to\Ri^L_{(1-\xi)\beta-\xi\eps}(\G)$ such that the following
diagram commutes up to homotopy:

\begin{equation}\label{eq:rips_comm_diagram}
\begin{tikzcd}
	{\Ri^L_{(1-\xi)\beta-\xi\varepsilon}(\G)} && {\Ri^\varepsilon_{\beta}(\s)} && {\Ri^L_{4\beta/3}(\G)} \\
	\\
	{\operatorname{sd} \K} && {\K}
	\arrow["g"', from=3-3, to=1-3]
	\arrow["{\widetilde{g}}", from=3-1, to=1-1]
	\arrow["{\sd}"', from=3-3, to=3-1]
	\arrow["\Phi", from=1-1, to=1-3]
	\arrow["\Psi", from=1-3, to=1-5]
	\arrow["{\Phi \circ \widetilde{g}}", from=3-1, to=1-3]
\end{tikzcd}
\end{equation}
where the linear homeomorphism $\sd^{-1}:\sd{\K}\longrightarrow\K$ maps each vertex of $\sd{\K}$ to the corresponding point of $\K$. 
Recall that each vertex of $\sd{\K}$ is the barycenter,
$\widehat{\sigma}_l$, of an $l$--simplex $\sigma_l$ of $\K$. In order to construct the simplicial map $\widetilde{g}$, we need to define it on the vertices of $\sd{\K}$ and prove that it extends to the simplicial map valued in 
$\Ri^L_{(1-\xi)\beta-\xi\varepsilon}(\G)$. 

Let $\sigma_l=[v_0,v_1,\ldots,v_l]$ be an $l$--simplex of $\K$. 
Since $g$ is a simplicial map, we have that the image
$g(\sigma_l)=[g(v_0),g(v_1),\ldots,g(v_l)]$ is a subset of $\s$ with 
$d^\eps_\s(g(v_i), g(v_j))<\beta$ for any  $0\leq i,j\leq l$. 
We note from (\ref{eq:diam-estimate-2}) that $\diam_\G \bigl(\Psi\circ g(\sigma_l)\bigr)<\frac{\ell(\G)}{3}$. 
Proposition~\ref{prop:center} implies that the circumcenter $c((\Psi\circ g)(\sigma_l))$ exists, and we define the value of $\widetilde{g}$ at the barycenter $\widehat{\sigma}_l$.
\[
\widetilde{g}(\widehat{\sigma}_l)\coloneqq c((\Psi\circ g)(\sigma_l)).
\] 
For any $0\leq j\leq l$, it also follows from Proposition~\ref{prop:radius} that given $\sigma_j\prec\sigma_l$:
\begin{equation}\label{eq:sigma-1}
d^L_\G\bigl(c(\Psi\circ g(\sigma_j)), c(\Psi\circ g(\sigma_l))\bigr)\leq\tfrac{1}{2}\diam^L_\G{(\Psi\circ g)(\sigma_l)}.
\end{equation} 
Now, to see that $\widetilde{g}$ extends to a simplicial map, consider a typical $l$--simplex $\tau_l=[\widehat{\sigma}_0$, $\widehat{\sigma}_1$, $\ldots$, $\widehat{\sigma}_l]$ of $\sd\K$, where $\sigma_{i-1}\prec\sigma_i$ for $1\leq i\leq l$ and $\sigma_i\in\K$, we have
\begin{align*}
\diam_\G\bigl(\widetilde{g}(\tau_l)\bigr)
&=\diam_\G\bigl\{\widetilde{g}(\widehat{\sigma_0}), \widetilde{g}(\widehat{\sigma_1}),\ldots,\widetilde{g}(\widehat{\sigma_l})\bigr\}  \\
&=\max_{1\leq i< j\leq l}{d_\G(\widetilde{g}(\widehat{\sigma_i}), \widetilde{g}(\widehat{\sigma_j}))} \leq\max_{1\leq j\leq l}{\frac{1}{2}\diam_\G{(\Psi\circ g)(\sigma_j})} \leq\tfrac{1}{2}\diam_\G{(\Psi\circ g)(\sigma_l}) \\
& \leq \frac{1}{2}\cdot\frac{1+2\xi}{1+\xi}(\beta+\xi\eps),\text{ from (\ref{eq:diam-estimate-1})}\\
&\leq(1-\xi)\beta-\xi\eps,\text{ since }\quad\eps\leq\beta/3\text{ and }0<\xi<1/4.    
\end{align*}
Thus, $\widetilde{g}(\tau_l)$ is a simplex of $\Ri^L_{(1-\xi)\beta-\xi\eps}(\G)$, which implies that $\widetilde{g}$ is a simplicial map.  
Naturally, $\Phi\circ\widetilde{g}: \operatorname{sd}(\K) \rightarrow \Ri^\eps_\beta(\s)$ is a simplicial map. Next, we intend to apply Lemma~\ref{lem:barycenter-map-lemma}, to the subdiagram of \eqref{eq:rips_comm_diagram} corresponding to maps $\Phi \circ \widetilde{g}$, $\sd$ and $g$. The assumption (1) of Lemma \ref{lem:barycenter-map-lemma} holds by definition of $\widetilde{g}$, we need to show the second condition (2) holds as well i.e. show that vertices in $g(\sigma_l) \cup \{(\Phi\circ\widetilde{g})(\widehat{\sigma_l})\}$ form a simplex $[(\Phi\circ\widetilde{g})(v_0), (\Phi\circ\widetilde{g})(v_1), \ldots,(\Phi\circ\widetilde{g})(v_l), (\Phi\circ\widetilde{g})(\widehat{\sigma}_l)]$ in $\Ri^\eps_\beta(\s)$. Observe
\begin{align*}
d^\eps_\s((\Phi\circ \widetilde{g})(v_i), (\Phi\circ\widetilde{g})(\widehat{\sigma}_l)) \leq \dfrac{d_\G(\widetilde{g}(v_i), \widetilde{g}(\widehat{\sigma}_l)) + \xi\eps}{(1-\xi)} < \beta,
\end{align*}
since $d_\G(\widetilde{g}(v_i), \widetilde{g}(\widehat{\sigma}_l))<(1-\xi)\beta-\xi\eps$ by the fact that $\widetilde{g}$ is simplicial. Since $g(v_i)=(\Phi\circ\widetilde{g})(v_i)$
\[[g(v_0), g(v_1), \ldots, g(v_l), (\Phi\circ\widetilde{g})(\widehat{\sigma}_l)]\]
is a simplex in $\Ri^\eps_{\beta}(\s)$.
Thus the assumptions of Lemma~\ref{lem:barycenter-map-lemma} are satisfied implying implies that $g$ and $\Phi\circ\tilde{g}$ induce homotopic maps and 
\[
\pi_\ast(\Phi)([\widetilde{g}])=[\Phi\circ\widetilde{g}]=[g],
\]
which yields the claim that $\pi_\ast(\Phi)$ is an epimorphism. 

Applying Whitehead's Theorem yields the desired homotopy equivalence 
$\Ri^\eps_\beta(\s)\simeq \G$.
\end{proof}
Fixing $\xi=1/6$, we get a simpler but weaker statement.
\begin{corollary}[Simpler Homotopy Reconstruction]\label{cor:rips}
Let $\G \subset \mathbb{R}^N$ be a compact, connected metric graph.  
For any positive $\beta<\ell(\G)/4$, choose a positive $\eps\leq\beta/3$ such that $\delta^{\eps}_{\beta}(\G)\leq8/7$. 
If $\s\subset \R^N$ is compact and $d_H(\G,\s)<\eps/12$, then we have a homotopy equivalence $\Ri^\eps_\beta(\s)\simeq\G$.
\end{corollary}

\section{Shadows of Simplicial Complexes}\label{sec:shadow}
Let $\K$ be an abstract simplicial complex with vertices in $\R^N$, i.e., $\K^{(0)}\subset\R^N$. 
In this section, we define the shadow (or \emph{geometric projection}) of $\K$ as a subset of $\R^N$ and study the natural shadow projection map. 

The \emph{shadow projection} map $p\colon|\K|\to\mathbb{R}^N$ sends a vertex $v\in\K^{(0)}$ to the corresponding point in $\mathbb{R}^N$ then extends linearly to all points of the geometric realization $|\K|$.
Note that $p$ is continuous.
\begin{definition}[Shadow]
We define the \emph{shadow} of $\K$ as its image under the projection map $p$, i.e.,
\[
\sh(\K)\coloneqq
\bigcup_{\sigma=[v_0,v_1,\ldots,v_k]\in\K}\mathrm{Conv}(\sigma),
\]
where $\mathrm{Conv}(\cdot)$ denotes  the convex hull of a subset in $\mathbb{R}^N$. 
\end{definition}
Since the shadow is a polyhedral subset of $\R^N$, it can be realized by the $N$-dimensional skeleton of $\K$ by Carath\'eodory’s theorem~\cite{Eckhoff1993-kf}.
We now describe a special simplicial complex decomposition of the \emph{shadow} in $\R^2$; see Figure~\ref{fig:shadow}.
We call it the \emph{shadow complex} and denote it by $\sc(\K)$.
\begin{enumerate}[(A)]
\item A \emph{shadow vertex} $v\in\sc(\K)$ is either $v\in\K^{(0)}$ or a transverse intersection of $p(e_1)$ and $p(e_2)$ for edges $e_1,e_2\in\K^{(1)}$. The transverse intersections are shown in red in Figure~\ref{fig:shadow}.
\item Triangulate the planar shadow using the shadow vertices such that a shadow edge or face does not contain any other vertices of the shadow. 
Consequently, the realization of any shadow simplex $\sigma\in\sc(\K)$ is contained in $p(\tau)$ for some $\tau\in\K$.
\end{enumerate}

\begin{remark}
Note that the triangulation described by (B) above may not be unique. 
We abuse notation here to denote any such triangulation by $\sc(\K)$. 
\end{remark}

\begin{figure}[htb]
\centering
\begin{subfigure}[t]{0.3\textwidth}
\centering
\begin{tikzpicture}[scale=1.2]
\path[name path=1] (3, 1.5) -- (0,1);
\filldraw[thick, fill=blue!10, draw=blue!50, name path=2] (0, 0) -- (2,0) -- (1,2) -- cycle;
\path[name path=3] (2.5, 2) -- (1.2,0.4);
\draw[thick, draw=blue!50, name intersections={of=1 and 2}] (3, 1.5) -- (intersection-1);
\draw[thick, draw=blue!50, name intersections={of=1 and 2}, dashed] (intersection-1) -- (intersection-2);
\draw[thick, draw=blue!50, name intersections={of=1 and 2}] (intersection-2) -- (0,1);
\draw[dashed, thick, draw=blue!50, name intersections={of=2 and 3}] (1.2, 0.4) -- (intersection-1);
\draw[thick, draw=blue!50, name intersections={of=2 and 3, name=E}, name intersections={of=1 and 3, name=F},shorten >=0.1cm] (E-1) -- (F-1);
\draw[thick, draw=blue!50, name intersections={of=1 and 3, name=F}, shorten <=0.1cm] (F-1) -- (2.5, 2);
\filldraw[fill=blue!10, draw=blue!80] (0, 0) circle (2pt);
\filldraw[fill=blue!10, draw=blue!80] (2, 0) circle (2pt);
\filldraw[fill=blue!10, draw=blue!80] (1, 2) circle (2pt);
\filldraw[fill=blue!10, draw=blue!80] (0, 1) circle (2pt);
\filldraw[fill=blue!10, draw=blue!80] (3, 1.5) circle (2pt);
\filldraw[fill=blue!10, draw=blue!80] (2.5, 2) circle (2pt);
\filldraw[fill=blue!10, draw=blue!80] (1.2, 0.4) circle (2pt);
\end{tikzpicture}
\end{subfigure}%
~ 
\begin{subfigure}[t]{0.3\textwidth}
\centering
\begin{tikzpicture}[scale=1.2]
\path[name path=1] (3, 1.5) -- (0,1);
\filldraw[thick, fill=blue!10, draw= blue!30, name path=2] (0, 0) -- (2,0) -- (1,2) -- cycle;
\path[name path=3] (2.5, 2) -- (1.2,0.4);
\draw[thick, draw=blue!30, name intersections={of=1 and 2}] (3, 1.5) -- (intersection-1);
\draw[thick, draw=blue!30, name intersections={of=1 and 2}] (intersection-2) -- (0,1);
\draw[thick, draw=blue!30, name intersections={of=2 and 3, name=E}, name intersections={of=1 and 3, name=F},shorten >=0.1cm] (E-1) -- (2.5,2);
\end{tikzpicture}
\end{subfigure}%
~
\begin{subfigure}[t]{0.3\textwidth}
\centering
\begin{tikzpicture}[scale=1.2]
\path[name path=1] (3, 1.5) -- (0,1);
\filldraw[thick, fill=blue!10, draw=blue!50, name path=2] (0, 0) -- (2,0) -- (1,2) -- cycle;
\path[name path=3] (2.5, 2) -- (1.2,0.4);
\draw[thick, draw=blue!50, name intersections={of=1 and 2}] (3, 1.5) -- (intersection-1);
\draw[thick, draw=blue!50, name intersections={of=1 and 2}] (intersection-1) -- (intersection-2);
\draw[thick, draw=blue!50, name intersections={of=1 and 2}] (intersection-2) -- (0,1);
\draw[thick, draw=blue!50, name intersections={of=2 and 3}] (1.2, 0.4) -- (intersection-1);
\draw[thick, draw=blue!50, name intersections={of=2 and 3, name=E}, name intersections={of=1 and 3, name=F},shorten >=0.1cm] (E-1) -- (F-1);
\draw[thick, draw=blue!50, name intersections={of=1 and 3, name=F}, shorten <=0.1cm] (F-1) -- (2.5, 2);
\draw[thick, draw=blue!50, name intersections={of=1 and 2, name=E}, name intersections={of=1 and 3, name=F}] (E-2) -- (1.2, 0.4);
\draw[thick, draw=blue!50, name intersections={of=1 and 2, name=E}, name intersections={of=1 and 3, name=F}] (E-1) -- (1.2, 0.4);
\draw[thick, draw=blue!50] (1.2, 0.4) -- (0,0);
\draw[thick, draw=blue!50] (1.2, 0.4) -- (2,0);
\filldraw[fill=blue!10, draw=blue!80] (0, 0) circle (2pt);
\filldraw[fill=blue!10, draw=blue!80] (2, 0) circle (2pt);
\filldraw[fill=blue!10, draw=blue!80] (1, 2) circle (2pt);
\filldraw[fill=blue!10, draw=blue!80] (0, 1) circle (2pt);
\filldraw[fill=blue!10, draw=blue!80] (3, 1.5) circle (2pt);
\filldraw[fill=blue!10, draw=blue!80] (2.5, 2) circle (2pt);
\filldraw[fill=blue!10, draw=blue!80] (1.2, 0.4) circle (2pt);
\filldraw[fill=red!10, draw=red!80,name intersections={of=1 and 2}] (intersection-1) circle (2pt);
\filldraw[fill=red!10, draw=red!80,name intersections={of=1 and 2}] (intersection-2) circle (2pt);
\filldraw[fill=red!10, draw=red!80,name intersections={of=1 and 3}] (intersection-1) circle (2pt);
\filldraw[fill=red!10, draw=red!80,name intersections={of=2 and 3}] (intersection-1) circle (2pt);
\end{tikzpicture}
\end{subfigure}
\caption{[Left] An abstract simplicial complex $\K$ with planar vertices has been depicted. [Middle] The shadow $\sh(\K)$ has been shown as a subset of the plane. 
[Right] The shadow complex $\sc(\K)$ has been drawn. 
The new shadow vertices due to transverse intersection are shown in red.}
\label{fig:shadow}
\end{figure}

We use the following notation throughout the rest of the paper.
We denote the simplices of an abstract simplicial complex $\K$ by square braces, e.g., $[ABC]$ for an abstract simplex on vertices $A, B, C\in\R^N$. 
We denote the convex-hull of Euclidean points without any adornment, e.g., $ABC$ denotes the Euclidean (filled in) triangle formed by the vertices $A,B,C$. Lastly, the Euclidean length of a line segment $AB$ is specified by $\overline{AB}$.

\subsection{$\pi_1$-epimorphism in $\R^2$}

We provide a sufficient condition for the abstract simplicial complex $\K$ so that $p$ induces an epimorphism on the fundamental group.

As already described in the introduction, we mention here the works of \cite{Chambers2010}, where the authors considered the homotopy equivalence of the shadow projection in the particular case when $\K$ is the (Euclidean) Vietoris--Rips complex of a finite, planar point set.
The projection was shown to induce $\pi_1$-isomorphism.
We generalize such results to any simplicial complex $\K$ by proposing a general \emph{lifting condition}  to ensure $\pi_1$-epimorphism of the projection map $p$.
For geometric graph reconstruction, we apply the condition to the Vietoris--Rips complex of planar samples but under a possibly non-Euclidean metric: the $\eps$-path metric.
In particular, we infer below that when $\K$ is the Vietoris--Rips complex under the Euclidean metric (as considered in \cite{Chambers2010}), $\K$ automatically satisfies the lifting condition.

Let $\K$ be a simplicial complex with vertices in $\R^2$.
Any path $\gamma$ in the geometric realization of its $1$-skeleton $|\K^{(1)}|$ can be described by a sequence of oriented edges in $\K$, and its projection $p(\gamma)$ must also form a path in the $1$-skeleton of the shadow $\sh(\K)$, such paths are further referred to as \emph{shadow paths}.
However, the converse is not generally true, i.e., every shadow path is not necessarily the projection of a path in the geometric realization of $\K$. 
Nonetheless, every shadow path can be \emph{lifted up to homotopy} to $\K$ under the following lifting condition.
\begin{definition}[Shadow Path Lifting]\label{def:lifting-cond}
We say that an abstract simplicial complex $|\K|$ with $\K^{(0)}\subset\R^2$ satisfies the \emph{lifting condition} up to homotopy if whenever the images $AB$ and $CD$ of two edges $[AB]$ and $[CD]$ of $\K$ intersect, then there exist vertices $E,F\in\K^{(0)}$ such that either 
\begin{enumerate}[(A)]
\item $[ABE]$ and $[CDE]$ are simplices of $\K$, 
\item[] or 
\item $[AEF]$ and $[CDEF]$ are simplices of $\K$ with $EF$ intersecting $AB$.
\end{enumerate}
\end{definition}
We remark that in the setting where $\K$ is the Euclidean Vietoris--Rips complex (as considered in \cite{Chambers2010}), one can choose $E$ to be one of the four initial vertices that is closest to the intersection of $AB$ and $CD$ to meet condition (A).

The following theorem guarantees that $p$ induces epimorphism on the fundamental group of $\K$.
\begin{theorem}[$\pi_1$--epimorphism]\label{thm:shadow-surjection}
Let $\K$ satisfy the above lifting condition. 
Then, the projection map $p$ induces an epimorphism on the fundamental groups.
\end{theorem}
\begin{proof}
Every oriented shadow path $\gamma$ in $\sh(\K)$ can be associated (up to homotopy) with a sequence $\widetilde{\gamma}$ of oriented edges in $\K$. Note that these edges from $\K$ may not necessarily form a path, but projections of consecutive edges must intersect at a shadow vertex.
If none of the intersections are transverse, there is nothing to prove.

\begin{figure}[htb]
\centering
\begin{subfigure}[t]{0.45\textwidth}
\centering
\begin{tikzpicture}[scale=0.8]
\coordinate (A) at (-2, 0);
\coordinate (B) at (2, 0);
\coordinate (C) at (0.5, -2);
\coordinate (D) at (1.5, 1.5);
\coordinate (E) at (-1.2, -1.5);
\coordinate (F) at (-1, 1);

\fill[fill=blue!20, opacity=1] (A) -- (E) -- (B) -- cycle;
\draw[thick, draw=blue!70, name path=1] (A) -- (B);
\draw[thick, draw=blue!70, name path=3] (D) -- (C);
\draw[thick, draw=blue!70] (E) -- (B);
\draw[thick, draw=blue!70] (D) -- (E);
\draw[thick, draw=blue!70] (E) -- (C);
\draw[thick, draw=blue!70] (A) -- (E);
\filldraw[fill=blue!10, draw=blue!80] (A) circle (2pt) node[anchor=east] {$A$};
\filldraw[fill=blue!10, draw=blue!80] (B) circle (2pt) node[anchor=west] {$B$};
\filldraw[fill=blue!10, draw=blue!80] (D) circle (2pt) node[anchor=south] {$D$};
\filldraw[fill=blue!10, draw=blue!80] (C) circle (2pt) node[anchor=north] {$C$};
\filldraw[fill=blue!10, draw=blue!80] (E) circle (2pt) node[anchor=north] {$E$};
\filldraw[fill=red!10, draw=red!80, name intersections={of=1 and 3}] (intersection-1) circle (2pt) node[anchor=south west, xshift=1.2ex] {$v$};
\end{tikzpicture}
\end{subfigure}%
~ 
\begin{subfigure}[t]{0.45\textwidth}
\centering
\begin{tikzpicture}[scale=0.8]
\coordinate (A) at (-2, 0);
\coordinate (B) at (2, 0);
\coordinate (C) at (0.5, -2);
\coordinate (D) at (1.5, 1.5);
\coordinate (E) at (-1.2, -1.5);
\coordinate (F) at (-1, 1);

\fill[fill=blue!20, opacity=1] (E) -- (C) -- (D) -- cycle;
\draw[thick, draw=blue!70, name path=1] (A) -- (B);
\draw[thick, draw=blue!70, name path=3] (D) -- (C);
\draw[thick, draw=blue!70] (E) -- (B);
\draw[thick, draw=blue!70] (D) -- (E);
\draw[thick, draw=blue!70] (E) -- (C);
\draw[thick, draw=blue!70] (A) -- (E);
\filldraw[fill=blue!10, draw=blue!80] (A) circle (2pt) node[anchor=east] {$A$};
\filldraw[fill=blue!10, draw=blue!80] (B) circle (2pt) node[anchor=west] {$B$};
\filldraw[fill=blue!10, draw=blue!80] (D) circle (2pt) node[anchor=south] {$D$};
\filldraw[fill=blue!10, draw=blue!80] (C) circle (2pt) node[anchor=north] {$C$};
\filldraw[fill=blue!10, draw=blue!80] (E) circle (2pt) node[anchor=north] {$E$};
\filldraw[fill=red!10, draw=red!80, name intersections={of=1 and 3}] (intersection-1) circle (2pt) node[anchor=south west, xshift=1.2ex] {$v$};
\end{tikzpicture}
\end{subfigure}%
\caption{Case A}
\label{fig:shadow-A}
\end{figure}

Let us, therefore, assume that $\widetilde{\gamma}$ contains two consecutive oriented edges $[AB]$ and $[CD]$ of $\K$ that transversely intersect at $v$ in $\sh(\K)$.
Without any loss of generality, it suffices to show that $AvD$ in $\gamma$ can be replaced with a sequence of edges of $\K$ starting at $A$ and ending at $D$, consecutively intersecting in $\K^{(0)}$ and homotopic to $AvD$ in $\sh(\K)$ w.r.t. the endpoints.
We now construct such sequences for $AvD$, considering each of the lifting conditions (A) and (B).\\

\noindent\textbf{Case A.}
If condition (A) is satisfied (Figure~\ref{fig:shadow-A}), then we replace $AvD$ with the sequence $\{AE, ED\}\subset\K$:
\begin{align*}
AED &\simeq (AE)D\simeq (AvE)D,~\text{Figure~\ref{fig:shadow-A}(a)} \\
&\simeq A(vED)\simeq A(vD),~\text{Figure~\ref{fig:shadow-A}(b)} \\
&\simeq AvD.
\end{align*}
Using the above homotopy, note that $AvC$, $BvD$, and $BvC$ can similarly be replaced with a sequence.

\noindent\textbf{Case B.}
If condition (B) is satisfied (Figure~\ref{fig:shadow-B}), we denote by $w$ the intersection of $EF$ and $AB$. Then, we replace $AvD$ with the sequence $\{AF, FD\}\subset\K$:
\begin{align*}
AFD &\simeq (AF)D \simeq (AwF)D,~\text{Figure~\ref{fig:shadow-B}(a)} \\
&\simeq A(wF)D \simeq A(wvF)D\simeq Aw(vFD) \simeq Aw(vD),~\text{Figure~\ref{fig:shadow-B}(b)} \\
&\simeq (Awv)D\simeq AvD.
\end{align*}
\begin{figure}[htb]
\centering
\begin{subfigure}[t]{0.45\textwidth}
\centering
\begin{tikzpicture}[scale=0.8]
\coordinate (A) at (-2, 0);
\coordinate (B) at (2, 0);
\coordinate (C) at (0.5, -2);
\coordinate (D) at (1.5, 1.5);
\coordinate (E) at (-1.2, -1.5);
\coordinate (F) at (-1, 1);

\fill[fill=blue!20, opacity=1] (A) -- (E) -- (F) -- cycle;
\draw[thick, draw=blue!70, name path=1] (A) -- (B);
\draw[thick, draw=blue!70, name path=3] (D) -- (C);
\draw[thick, draw=blue!70] (D) -- (F);
\draw[thick, draw=blue!70] (F) -- (C);
\draw[thick, draw=blue!70] (D) -- (E);
\draw[thick, draw=blue!70] (E) -- (C);
\draw[thick, draw=blue!70] (A) -- (E);
\draw[thick, draw=blue!70, name path=2] (E) -- (F);
\draw[thick, draw=blue!70] (A) -- (F);
\filldraw[fill=blue!10, draw=blue!80] (A) circle (2pt) node[anchor=east] {$A$};
\filldraw[fill=blue!10, draw=blue!80] (B) circle (2pt) node[anchor=west] {$B$};
\filldraw[fill=blue!10, draw=blue!80] (D) circle (2pt) node[anchor=south] {$D$};
\filldraw[fill=blue!10, draw=blue!80] (C) circle (2pt) node[anchor=north] {$C$};
\filldraw[fill=blue!10, draw=blue!80] (F) circle (2pt) node[anchor=east] {$F$};
\filldraw[fill=blue!10, draw=blue!80] (E) circle (2pt) node[anchor=north] {$E$};
\filldraw[fill=red!10, draw=red!80, name intersections={of=1 and 2}] (intersection-1) circle (2pt) node[anchor=north west] {$w$};
\filldraw[fill=red!10, draw=red!80, name intersections={of=1 and 3}] (intersection-1) circle (2pt) node[anchor=south west, xshift=1.2ex] {$v$};
\end{tikzpicture}
\end{subfigure}%
~ 
\begin{subfigure}[t]{0.45\textwidth}
\centering
\begin{tikzpicture}[scale=0.8]
\coordinate (A) at (-2, 0);
\coordinate (B) at (2, 0);
\coordinate (C) at (0.5, -2);
\coordinate (D) at (1.5, 1.5);
\coordinate (E) at (-1.2, -1.5);
\coordinate (F) at (-1, 1);

\fill[fill=blue!20, opacity=1] (F) -- (E) -- (C) -- (D) -- cycle;
\draw[thick, draw=blue!70, name path=1] (A) -- (B);
\draw[thick, draw=blue!70, name path=3] (D) -- (C);
\draw[thick, draw=blue!70] (D) -- (F);
\draw[thick, draw=blue!30] (F) -- (C);
\draw[thick, draw=blue!30] (D) -- (E);
\draw[thick, draw=blue!70] (E) -- (C);
\draw[thick, draw=blue!70] (A) -- (E);
\draw[thick, draw=blue!70, name path=2] (E) -- (F);
\draw[thick, draw=blue!70] (A) -- (F);
\filldraw[fill=blue!10, draw=blue!80] (A) circle (2pt) node[anchor=east] {$A$};
\filldraw[fill=blue!10, draw=blue!80] (B) circle (2pt) node[anchor=west] {$B$};
\filldraw[fill=blue!10, draw=blue!80] (D) circle (2pt) node[anchor=south] {$D$};
\filldraw[fill=blue!10, draw=blue!80] (C) circle (2pt) node[anchor=north] {$C$};
\filldraw[fill=blue!10, draw=blue!80] (F) circle (2pt) node[anchor=east] {$F$};
\filldraw[fill=blue!10, draw=blue!80] (E) circle (2pt) node[anchor=north] {$E$};
\filldraw[fill=red!10, draw=red!80, name intersections={of=1 and 2}] (intersection-1) circle (2pt) node[anchor=north west] {$w$};
\filldraw[fill=red!10, draw=red!80, name intersections={of=1 and 3}] (intersection-1) circle (2pt) node[anchor=south west, xshift=1.2ex] {$v$};
\end{tikzpicture}
\end{subfigure}%
\caption{Case B}
\label{fig:shadow-B}
\end{figure}

Using the above homotopy, note that $AvC$, $BvD$, and $BvC$ can similarly be replaced with a sequence.
\end{proof}

\section{Geometric Reconstruction using Vietoris--Rips Shadow}\label{sec:geom}
This section considers the geometric reconstruction of a Euclidean graph $\G$ from a noisy Euclidean sample $\s$ using the shadow of Vietoris--Rips complexes.
We assume that both the graph and sample are hosted in the plane, i.e., $N=2$.
For homotopy reconstruction in Section~\ref{sec:top}, $\G$ is assumed to be only compact and connected.
However, for geometric reconstruction, we impose a few more \emph{geometric regularity} assumptions on $\G$.

\subsection{Assumptions for Geometric Reconstruction}\label{sec:assumptions}
\begin{enumerate}
\item[(A1)] $\G$ is compact and connected with $|\mathscr{E}(\G)|<\infty$;
\item[(A2)] any two edges of $\G$ are incident to at most one common vertex;
\item[(A3)] each (open) edge of $\G$ is at least $C^1$;
\item[(A4)] the tangents of each pair of incident edges $e_1,e_2\in\mathscr{E}(\G)$ of $\G$ make an angle in $(0,\pi]$, and is denoted by $\angle{e_1e_2}$.
We set $\angle{e_1e_2}=\infty$ if $e_1,e_2$ are not incident.
\end{enumerate}

\subsection{Shadow Radius}
For a graph $\G$ having properties (A1--A4), we first define the following quantity:
\begin{equation}\label{eq:theta}
\Theta\coloneqq
\max\left\{\frac{1}{2},\cos^2{\left(\frac{1}{2}\min_{e_1,e_2\in \mathscr{E}(\G)}\angle{e_1e_2}\right)}\right\}.
\end{equation}
Clearly, $1/2\leq\Theta<1$, due to the fact that $\angle{e_1e_2}\in(0, \pi]$ for $e_1, e_2\in\mathscr{E}(\G)$.
In the trivial case, where $\mathscr{E}(\G)$ is singleton, we take $\Theta=1/2$.

We are now ready to define the shadow radius.
\begin{definition}[Shadow Radius]\label{def:shadow}
Let $\G\subset\R^2$ be a graph. 
The \emph{shadow radius} of $\G$, denoted $\Delta(\G)$, is defined as the least upper bound for $r\geq0$ satisfying the following property:

($\star$) For $a, b\in \G$ with $\|a-b\|\leq r$ if $q\in \G\cap ab^\eps$ for some $0\leq\eps\leq\frac{1}{2}[1-\Theta]^{3/2}r$, then 
\[\min\left\{d_\G(a, q),d_\G(b,q)\right\}\leq\frac{1+\Theta}{2} d_\G(a,b)+\frac{\eps}{1-\Theta}.\]
\end{definition}
Here, $ab^\eps$ denotes the Euclidean $\eps$-thickening of the segment $ab$.
\begin{figure}[ht]
\centering
\includegraphics[width=0.35\linewidth]{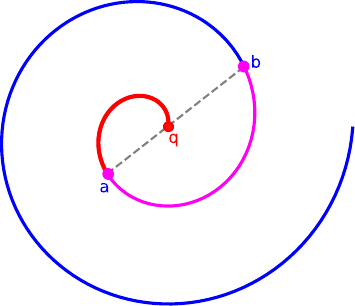}
\medskip
\caption{For a spiral shaped graph $\G$, the quantity $\Theta=1/2$ (since all edge tangents at vertices meet at angle $\pi$, the second term in the definition of $\Theta$ vanishes in~\ref{eq:theta}), and the condition ($\star$) in Definition~\ref{def:shadow} may not satisfy for all $a,b\in \G$. 
However, if $d_\G(a,b)$ is small, i.e., less than $\Delta(\G)$, the condition is always satisfied, as shown.}
\label{fig:spiral-Delta(G)}
\end{figure}
The above property ($\star$) simply says if a point $q\in\G$ lies in the close Euclidean proximity to the segment joining points $a,b\in\G$, then $q$ must be within `controlled' geodesic proximity to either $a$ or $b$; see Figure~\ref{fig:spiral-Delta(G)}.

If $\G$ satisfies assumptions (A1--A4), we show that its shadow radius is positive.
See Appendix for a proof.
\begin{proposition}[Positivity of Shadow Radius]\label{prop:shadow-radius}
For a planar graph $\G$ having properties (A1--A4), it satisfies $\Delta(\G)>0$.
\end{proposition}

\begin{remark}[Geometric content of $\Delta(\G)$]
The positivity of $\Delta(\G)$ established in Proposition~\ref{prop:shadow-radius} is tied to two concrete geometric features of $\G$. First, the condition $\Theta < 1$ requires that no two incident edges of $\G$ form a cusp at a common vertex, i.e., every pair of incident edges meets at a strictly positive angle. 
Second, the $C^1$ regularity of edges (assumption $A3$) ensures that the tangent direction varies continuously along each edge, so the modulus of continuity of the tangent map is finite and controls how $\Delta(\G)$ can be bounded below near each vertex. More precisely, the proof of Proposition~\ref{prop:shadow-radius} shows that $\Delta(\G) \geq r$, where $r > 0$ is determined by the minimum incident angle $\min_{e_1, e_2 \in \mathcal E(\G)} \angle e_1 e_2$ and the local Lipschitz constant of the edge tangents near each vertex.
\end{remark}

\begin{remark}[Relation to $\mu$-reach \cite{chazal2006sampling}]
Assumption~(A4) and the condition $\Theta < 1$ are equivalent: both
require that every pair of incident edges meets at a strictly positive
angle $\alpha_{\min} > 0$. Note that (A4) allows the angle $\pi$
(edges meeting with opposite tangent vectors, a smooth junction), and
only rules out angle $0$ (edges sharing the same tangent direction at
a vertex, a cusp).
 
To see why (A1--A4) imply a positive $\mu$-reach for $\G$ as a whole,
recall that the \emph{classical reach} of $\G$ vanishes at any
vertex with opening angle strictly less than $\pi$, since the medial
axis of $\mathbb{R}^2 \setminus \G$ passes arbitrarily close to any
such vertex. The $\mu$-reach is more selective. At a medial axis
point $p$ with two nearest points $x_1, x_2 \in \G$, the generalized
gradient satisfies $\|\nabla_\G(p)\| = \cos (\angle x_1 p x_2 /2)$.
Near a vertex where two edges meet at opening angle $\alpha$, the
nearest points lie one on each edge and the angle $\angle x_1 p x_2
\approx \pi - \alpha$, giving $\|\nabla_\G(p)\| \approx \sin(\alpha/2)$.
The point $p$ is $\mu$-critical only if $\|\nabla_\G(p)\| \leq \mu$,
i.e., $\mu \geq \sin(\alpha/2)$. Setting $\alpha = \alpha_{\min} > 0$,
the medial axis near any vertex contains no $\mu$-critical points for
any $\mu < \sin(\alpha_{\min}/2)$. Together with the $C^1$ regularity
of edges~(A3) and compactness~(A1), this gives
$\mathrm{reach}_\mu(\G) > 0$ for $\mu < \sin(\alpha_{\min}/2)$,
for $\G$ as a whole.
 
Finally, the genuine complementarity between our approach and
$\mu$-reach-based methods lies in Section~3. Theorem~3.1 (homotopy
reconstruction) requires neither assumption~(A4) nor any positive
$\mu$-reach condition. It applies to compact connected graphs whose
$\mu$-reach may vanish, for instance when incident edges share a
tangent direction at a vertex ($\alpha_{\min} = 0$, angle $= 0$),
which forces $\Theta = 1$ and violates~(A4). For Theorem~3.1, the
only relevant parameters are the systole $\ell(\G)$ and the
large-scale distortion $\delta^\varepsilon_\beta(\G)$, with no
restriction on vertex angles.
\end{remark}

\subsection{Homotopy Equivalence}
We now present our $\pi_1$-isomorphism argument for the shadow projection.
We start with the following technical lemma, which we use later to ensure that the shadow lifting condition~(\ref{def:lifting-cond}) is satisfied for a sample near a Hausdorff-close graph $\G$.
The proof is given in the appendix.
\begin{proposition}[Intersection]
\label{prop:intersection} Let $\G \subset \mathbb{R}^2$ be a planar graph having properties (A1--A4) and $\s \subset \mathbb{R}^2$ with $d_H(\s,\G) < \frac{1}{2}\xi\varepsilon$ for some $\xi \in (0,1)$ and $0\leq\eps\leq\frac{1}{(1+\xi)}(1-\Theta)^{3/2}\Delta(\G)$.
If $A, B, C, D \in \s$ with $\max\{\overline{AB},\overline{CD}\}<\Delta(\G)$ and $AB$ intersecting $CD$, then there exist $E,F\in \s$ such that either
\begin{enumerate}[(A)]
\item $\max\{\diam^\eps_\s\{A, B, E\},\diam^\eps_\s\{C, D, E\}\}\leq\max\{d^\eps_\s(A,B),d^\eps_\s(C,D)\}$, or
\item $\diam^\eps_\s\{C, D, E, F\}\leq d^\eps_\s(C, D)$ and $EF$ intersects $AB$ with \[
\min\{\diam^\eps_\s\{A, E, F\},\diam^\eps_\s\{B, E, F\}\}\leq\frac{(1-\Theta^2) d_\G(a,b)+(3+\xi-2\Theta)\eps}{2(1-\xi)(1-\Theta)}.
\]
\end{enumerate}
\end{proposition}

We immediately note, using the triangle inequality ($E$ as a pivot), the following consequence.
\begin{corollary}[Intersection]\label{cor:intersection} 
Under the assumptions of Proposition~\ref{prop:intersection}, we have 
\[\diam^\eps_\s\{A, B, C, D\}\leq2\max\big\{d^\eps_\s(A,B),d^\eps_\s(C,D)\big\}
+\frac{(1-\Theta^2) d_\G(a,b)+(3+\xi-2\Theta)\eps}{2(1-\xi)(1-\Theta)}
.\]
\end{corollary}
Using the lifting condition introduced in Definition~\ref{def:lifting-cond}, we now show the $\pi_1$-surjectivity of the shadow projection.
\begin{lemma}[Surjectivity]\label{lem:surjection}
Let $\G \subset \mathbb{R}^2$ a graph having properties (A1--A4). 
Fix any $\xi\in\left(0,\frac{1-\Theta}{6}\right)$.
For any positive $\beta<\Delta(\G)$, choose\footnote{It is possible to choose such small $\eps$, as $\delta^\eps_\beta(\G)\to1$ as $\eps\to0$.} 
a positive $\eps\leq\frac{(1-\Theta)(1-\Theta-6\xi)}{12}\beta$ such that $\delta^{\eps}_{\beta}(\G)\leq\frac{1+2\xi}{1+\xi}$. 
If $\s\subset \R^2$ with $d_H(\G,\s)<\tfrac{1}{2}\xi\eps$, then the shadow projection of $\mathcal{K}=\Ri^\eps_\beta(\s)$ induces a surjective homomorphism on $\pi_1$. 
\end{lemma}
\begin{proof}
Let $[AB],[CD]\in\K$ such that $AB$ and $CD$ intersect in the shadow of $\K$.
From the definition of $\K$, we have $\max\{d^\eps_\s(A,B), d^\eps_\s(C,D)\}<\beta$; consequently, $\max\{\overline{AB},\overline{CD}\}<\beta$ from Proposition~\ref{prop:d^esp-d^L-estimate}.
Since $\beta<\Delta(\G)$, we get 
\begin{align*}
\eps &\leq\frac{(1-\Theta)(1-\Theta-6\xi)}{12}\beta
<\frac{(1-\Theta)^2}{12}\beta\leq\frac{(1-\Theta)^{3/2}}{12}\beta,\text{ since }\Theta<1 \\
&<\frac{(1-\Theta)^{3/2}}{1+\xi}\beta,
\text{ since }\xi<1\\
&<\frac{(1-\Theta)^{3/2}}{1+\xi}\Delta(\G),
\end{align*}
Proposition~\ref{prop:intersection} implies that there must exist $E, F\in \s$ such that 
\begin{enumerate}[(A)]
\item $\max\{\diam^\eps_\s\{A, B, E\},\diam^\eps_\s\{C, D, E\}\}\leq\max\{d^\eps_\s(A,B),d^\eps_\s(C,D)\}<\beta$, or
\item $\diam^\eps_\s\{C, D, E, F\}\leq d^\eps_\s(C, D)$ and $EF$ intersects $AB$ with 
\begin{equation}\label{eq:diameter-1}
\min\{\diam^\eps_\s\{A, E, F\},\diam^\eps_\s\{B, E, F\}\}\leq\tfrac{(1-\Theta^2) d_\G(a,b)+(3+\xi-2\Theta)\eps}{2(1-\xi)(1-\Theta)}.
\end{equation}
\end{enumerate}
These two conditions relate directly to (A), (B) of Definition~\ref{def:lifting-cond}.

In the case (B), we argue that $\min\{\diam^\eps_\s\{A, E, F\},\diam^\eps_\s\{B, E, F\}\}<\beta$.
This can be shown to hold true when assumed $d_\G(a,b)<\frac{2(1-\xi)(1-\Theta)\beta-(3+\xi-2\Theta)\eps}{1-\Theta^2}$, and plugged into the RHS of \eqref{eq:diameter-1}.
So, we consider the non-trivial case by assuming $d_\G(a,b)\geq\frac{2(1-\xi)(1-\Theta)\beta-(3+\xi-2\Theta)\eps}{1-\Theta^2}$ to get
\begin{align*}
d_\G(a,b)&\geq\frac{2(1-\xi)(1-\Theta)\beta-(3+\xi-2\Theta)\eps}{1-\Theta^2}\\
&\geq\frac{2(1-\xi)(1-\Theta)\beta-4\eps}{1-\Theta^2},
\text{ since }(3+\xi-2\Theta)\leq(3+\xi)\leq4\\
&\geq\frac{2(1-\xi)(1-\Theta)\beta-4\frac{(1-\Theta)(1-\Theta-6\xi)}{12}\beta}{1-\Theta^2}
,\text{ since }\eps\leq\frac{(1-\Theta)(1-\Theta-6\xi)}{12}\\
&=\frac{6(1-\xi)-(1-\Theta-6\xi)}{3(1+\Theta)}\beta
=\frac{5+\Theta}{3+3\Theta}\beta
>\beta,\text{ since }\Theta<1.
\end{align*}
Therefore, in this case, using Proposition~\ref{prop:delta_path} (for $R=\beta$ and $\alpha=\frac{1}{2}\xi\eps$), we get
\begin{align*}
&\min\{\diam^\eps_\s\{A, E, F\},\diam^\eps_\s\{B, E, F\}\} \\
&\leq\frac{(1-\Theta^2) d_\G(a,b)+(3+\xi-2\Theta)\eps}{2(1-\xi)(1-\Theta)}\\
&\leq\frac{(1-\Theta^2)\delta^\eps_\beta(\G)(\beta+\xi\eps)+(3+\xi-2\Theta)\eps}{2(1-\xi)(1-\Theta)}\\
&\leq\beta,\text{ since }\delta^\eps_\beta(\G)\leq\frac{1+2\xi}{1+\xi}\text{ and }\eps\leq\frac{(1-\Theta)(1-\Theta-6\xi)}{8}\beta.
\end{align*}
Consequently, $\K$ satisfies the lifting condition (Definition~\ref{def:lifting-cond}). 
Using Theorem~\ref{thm:shadow-surjection}, we then conclude the result.
\end{proof}
We now prove the $\pi_k$-injectivity of the shadow projection for any $k\geq0$.
\begin{lemma}[Injectivity]\label{lem:injection}
Let $\G \subset \mathbb{R}^2$ a graph having properties (A1--A4).
Fix any $\xi\in\left(0,\frac{1-\Theta}{6}\right)$.
For any positive $\beta<\min\left\{\Delta(\G),\frac{\ell(\G)}{18}\right\}$, choose 
a positive $\eps\leq\frac{(1-\Theta)(1-\Theta-6\xi)}{12}\beta$ such that $\delta^{\eps}_{\beta}(\G)\leq\frac{1+2\xi}{1+\xi}$. 
If $\s\subset \R^2$ with $d_H(\G,\s)<\tfrac{1}{2}\xi\eps$, then we the shadow projection of $\mathcal{K}=\Ri^\eps_\beta(\s)$ induces an injective homomorphism on $\pi_k$ for any $k\geq0$. 
\end{lemma}
\begin{proof}
Since $\xi\in\left(0,\frac{1}{4}\right)$ and $\eps\leq\beta/3$, we can reuse the idea of Diagram~\ref{eq:rips_comm_diagram} from the proof of Theorem~\ref{thm:hom-eq}:
\begin{equation}\label{eq:rips_shadow_diagram}
\begin{tikzcd}
{|\Ri^L_{(1-\xi)\beta-\xi\varepsilon}(\G)|} && {|\Ri^\varepsilon_{\beta}(\s)|} && {|\Ri^L_{6\beta}(\G)|} \\ \\
&& {|\sc(\Ri^\eps_\beta(\s))|} &&
\arrow["p"', from=1-3, to=3-3]
\arrow["|\Psi|", from=1-3, to=1-5]
\arrow["|\Phi|", from=1-1, to=1-3]
\arrow[blue, "|\Psi'|"', from=3-3, to=1-5]
\end{tikzcd}
\end{equation}
Here, $\sc(\Ri^\eps_\beta(\s))$ denotes the shadow \emph{complex} of $\Ri^\eps_\beta(\s)$ as defined in Definition~\ref{def:shadow}.
We first define a simplicial map $\Psi'\colon\sc(\Ri^\eps_\beta(\s))\to\Ri^L_{6\beta}(\G)$ such that Diagram~\ref{eq:rips_shadow_diagram} of geometric realizations commutes up to homotopy.

Recall that the vertex set of $\sc(\Ri^\eps_\beta(\s))$ is $\s\sqcup \mathcal{I}$, where $\mathcal{I}$ denotes the set of new vertices due to transverse intersections of edges of $\Ri^\eps_\beta(\s)$ in the shadow.

Define the vertex map $\Psi'\colon \s\to \G$ so that $\Psi'|_\s\coloneqq\Psi$. 
For any $I\in \mathcal{I}$, there must exist $1$-simplices $\sigma=[AB]$ and $\tau=[CD]$ of $\Ri^\varepsilon_{\beta}(\s)$ so that $I$ lies on the intersection of the two Euclidean segments $AB$ and $CD$. 
Define $\Psi'(P)\coloneqq\Psi(A)$; here the choice of $A$ is arbitrary.

Now, we show that the map extends to a simplicial map.
Let $\sigma=[U,V]\in\sc(\Ri^\eps_\beta(\s))$ be a shadow $1$-simplex to show that $[\Psi'(U), \Psi'(V)]$ is a simplex of $\Ri^L_{6\beta}(\G)$. 
If both $U,V\in \s$, then there is nothing to show since $\Psi'|_\s=\Psi$ is already a simplicial map. 

Otherwise, by the construction of the shadow complex, there exists (possibly degenerate) $[ABC]\in\Ri^\eps_\beta(\s)$ with $\max\{d^\eps_\s(A,U),d^\eps_\s(B,V)\}\leq3\beta$
by Corollary~\ref{cor:intersection}.
Since $d_\s^\eps(A,B)<\beta$, the triangle inequality implies $d^\eps_\s(U,V)<4\beta$.
Following the same (injectivity) argument given in the proof of Theorem~\ref{thm:hom-eq}, we get $d_\G(\Psi'(U), \Psi'(V))<6\beta$ using \eqref{eq:diam-estimate-2}.
So, $\Psi'$ is a simplicial map.

Since $\Psi'|_\s=\Psi$ and the restriction of shadow projection $p|_\s$ is the identity on $\s$, the triangle in the diagram of realizations commutes up to homotopy.
Using the fact that $6\beta<\ell(\G)/3$ and following the injectivity argument used in the proof of Theorem~\ref{thm:hom-eq}, we conclude that $\Psi$ induces injective homomorphisms on homotopy groups. 
Since the diagram commutes up to homotopy, $p$ must also induce surjections on $\pi_k$ for any $k\geq0$.
\end{proof}

We finally prove our main geometric graph reconstruction result.
\geom
\begin{proof}
We note that the above conditions satisfy the conditions of Theorem~\ref{thm:hom-eq}. Consequently, we already have $\Ri_\beta^\eps(\s)\simeq\G$.
Since the higher homotopy groups of $\Ri_\beta^\eps(\G)$ (being homotopy equivalent to $\G$) and $\sh(\Ri_\beta^\eps(\s))$ are trivial, Lemma~\ref{lem:surjection} and Lemma~\ref{lem:injection} together with Whitehead's theorem imply their homotopy equivalence.
Hence, $\sh(\Ri_\beta^\eps)\simeq\G$.

For any abstract simplicial complex $\K$ with Euclidean vertices, it holds that $d_H(\sh(\K), \K^{(0)})\leq\sup\left\{\overline{AB}\colon [AB]\in\K\right\}$.
In our case, it implies from Proposition~\ref{prop:d^esp-d^L-estimate} that \[d_H(\sh(\Ri_\beta^\eps(\s)), \s)
\leq\sup\left\{\overline{AB}\colon [AB]\in\Ri_\beta^\eps(\s)\right\}
\leq\sup\left\{d^\eps_S(A,B)\colon [AB]\in\Ri_\beta^\eps(\s)\right\}\leq\beta.\]
Since $d_H(\G,\s)<\frac{1}{2}\xi\eps$, from the triangle inequality we finally conclude  $d_H(\sh(\Ri_\beta^\eps(\s)), \G)\leq\left(\beta+\frac{1}{2}\xi\eps\right)$ as claimed.
\end{proof}

The following insightful corollary, gracefully suggested by the anonymous referee, follows immediately from the above theorem.
\begin{corollary}[$\tfrac{1}{2}\xi\varepsilon$-close homotopy equivalence]
\textit{Under the hypotheses of Theorem~\ref{thm:geom}, there exist continuous maps
\[
  f \colon \G \;\longrightarrow\; \sh(\mathcal R^\varepsilon_\beta(\s)),
  \qquad
  h \colon \sh(\mathcal R^\varepsilon_\beta(\s)) \;\longrightarrow\; \G
\]
such that $h \circ f \simeq \mathrm{Id}_\G$ and $f \circ h \simeq
\mathrm{Id}_{\sh(\mathcal R^\varepsilon_\beta(\s))}$, and moreover
\[
  \max\left\{\sup_{g \,\in\, \G} \|f(g) - g\|,
  \sup_{x \,\in\, \sh(\mathcal R^\varepsilon_\beta(\s))} \|h(x) - x\|
  \;\right\}<\; \tfrac{1}{2}\xi\varepsilon.
\]}
\end{corollary} 
The maps $f$ and $h$ are the nearest-point correspondences $\Phi\lvert_\G$ and $\Psi'$
already constructed in the proofs of Theorem~\ref{thm:hom-eq} and Lemma~\ref{lem:injection}; the displacement
bounds follow from their construction, with the linear extension to the shadow
handled by a convexity argument. The bound $\tfrac{1}{2}\xi\varepsilon$ is the same
constant that controls the Hausdorff proximity $d_H(\G, \s) < \tfrac{1}{2}\xi\varepsilon$
of the sample to the graph, which is natural.
 
As the referee noted, this is strictly stronger than stating homotopy equivalence and
Hausdorff closeness separately: two spaces may be homotopy equivalent and Hausdorff
close without admitting homotopy inverse maps that are pointwise close to the identity,
as illustrated by the referee's own Figure~\ref{fig:referee}. 
The corollary rules out such situations and
provides a more faithful geometric-topological correspondence between $\G$ and its
reconstruction $\sh(\mathcal R^\varepsilon_\beta(\s))$.

\begin{figure}[h]
\centering
\begin{tikzpicture}[line join=round, line cap=round, scale=0.8]
 
  \begin{scope}[xshift=-2.85cm]
    \fill[pattern=north east lines, pattern color=black, even odd rule]
      (0,0) circle (1.85)
      (0,0) circle (0.95);
    \fill[white] (0.94, -0.057) rectangle (1.86, 0.057);
    \draw[black, line width=0.7pt]
      (1.70:1.85) arc[start angle=1.70, end angle=358.30, radius=1.85];
    \draw[black, line width=0.7pt]
      (3.32:0.95) arc[start angle=3.32, end angle=356.68, radius=0.95];
    \draw[black, line width=0.7pt] (0.95, 0.057) -- (1.85, 0.057);
    \draw[black, line width=0.7pt] (0.95,-0.057) -- (1.85,-0.057);
    \fill[white] (-1.4, 0) circle (0.18);
    \draw[black, line width=0.5pt] (-1.4, 0) circle (0.18);
    \node[font=\small\itshape] at (0, -2.2) {$C$-shaped cut-washer with hole};
  \end{scope}
 
  \begin{scope}[xshift=2.85cm]
    \fill[pattern=north east lines, pattern color=black, even odd rule]
      (0,0) circle (1.85)
      (0,0) circle (0.95);
    \draw[black, line width=0.7pt] (0,0) circle (1.85);
    \draw[black, line width=0.7pt] (0,0) circle (0.95);
    \node[font=\small\itshape] at (0, -2.2) {round annulus};
  \end{scope}
\end{tikzpicture}
\caption{Two sets with the same homotopy type ($\simeq \mathbb S^1$) that are Hausdorff close, but are \textit{not} $\varepsilon$-close homotopy equivalent.\label{fig:referee}}
\end{figure}

\begin{figure}[thb]
\centering
\includegraphics[width=0.45\linewidth]{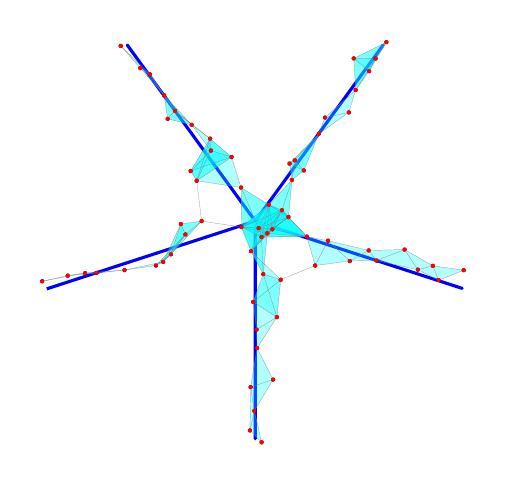}
\includegraphics[width=0.45\linewidth]{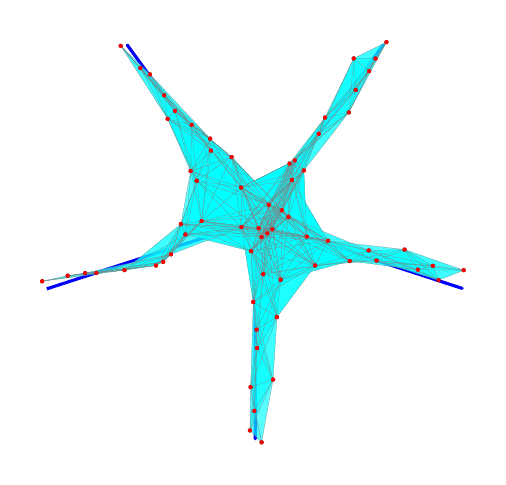}
\includegraphics[width=0.45\linewidth]{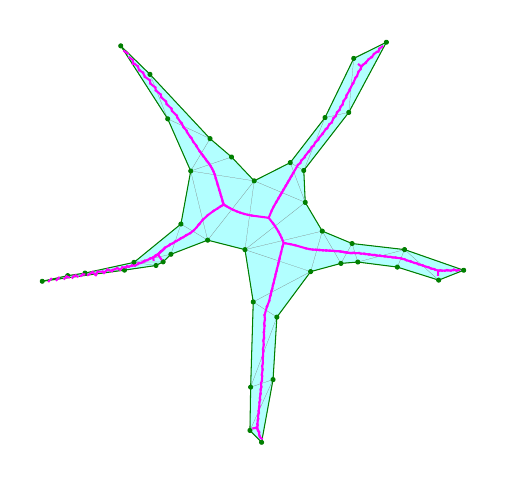}
\includegraphics[width=0.45\linewidth]{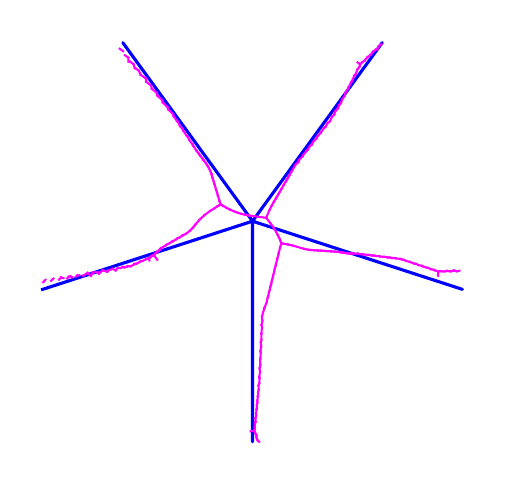}
\caption{Reconstruction of a $5$--pronged graph from a Hausdorff close sample via the shadow $\sh(\Ri_\beta^\eps(\s))$. The top left figure shows a shadow of the Euclidean Rips complex, which is topologically inaccurate (projected edges in grey and faces in light blue). The top right figure shows  $\sh(\Ri_\beta^\eps(\s))$, which reflects the correct homotopy type (Theorem \ref{thm:geom}). Further homotopy-equivalent simplifications of $\sh(\Ri_\beta^\eps(\s))$ are shown in the bottom-left and bottom-right figures. The bottom left shows the planar triangulation of the shadow $\sh(\Ri_\beta^\eps(\s))$ with marked green boundary. The purple geometric graph shows an approximation of the medial axis.}
\label{fig:5-prong-shadow-MA}
\end{figure}

\begin{figure}[htb]
\centering
\includegraphics[width=0.45\linewidth]{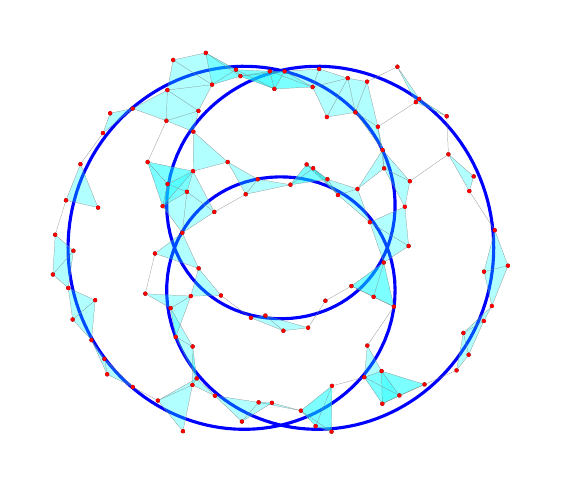}
\includegraphics[width=0.45\linewidth]{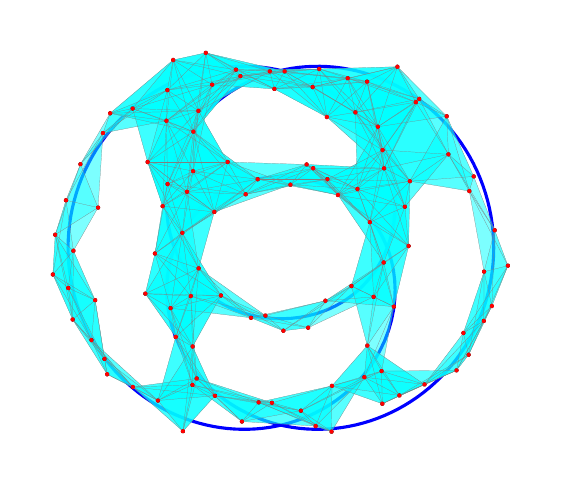}
\includegraphics[width=0.45\linewidth]{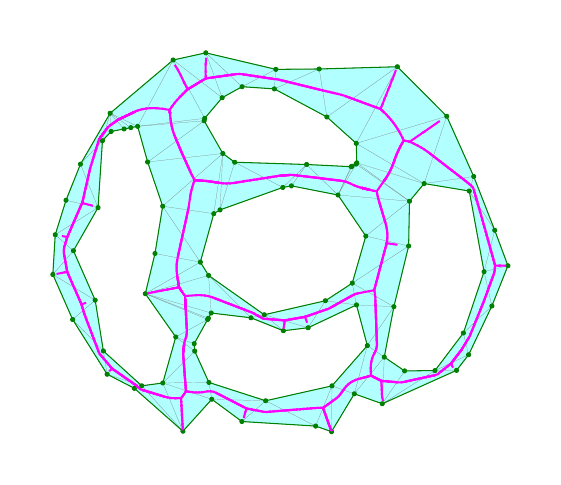}
\includegraphics[width=0.45\linewidth]{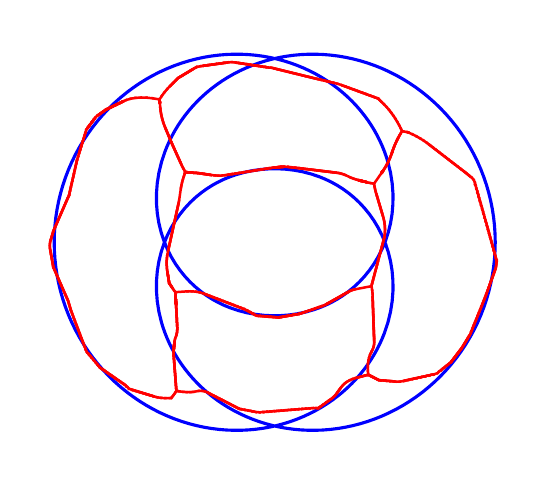}
\caption{Similar to Figure~\ref{fig:5-prong-shadow-MA}, here the geometric graph under reconstruction is a closed curve with no leaves. The [BOTTOM LEFT] panel shows the entire medial axis of the shadow complex $\sh$ (purple). The [BOTTOM RIGHT] panel displays a pruned version of the medial axis (red); see \cite{Attali2004Survey} for pruning strategies for medial axes.}
\label{fig:torus23-shadow-MA}
\end{figure}

\begin{remark}
In many settings, it is desirable to recover a $1$--dimensional proxy for the embedded graph $\G\subset\R^2$.
Since the shadow complex $\mathscr{Sh}=\sh(\Ri_\beta^\eps(\s))$ is a planar polygon, a natural choice for this proxy is 
the medial axis of the shadow; see \cite{Attali2004Survey} for a definition. 
In the planar case, this medial axis is itself a geometric graph, as shown in \cite{Choi1997}. 
In certain desirable cases, the shadow $\sh$ may be homotopy equivalent to its medial axis.
In such cases, the medial axis of $\sh$ will be a geometric graph homotopy equivalent to $\G$; 
see Figures~\ref{fig:5-prong-shadow-MA} and \ref{fig:torus23-shadow-MA} for an illustration.
\end{remark}


\section{Discussions}
The current work successfully provides guarantees for a homotopy-type recovery of an embedded metric graph in $\R^N$ from the Vietoris-Rips complexes of a Hausdorff-close Euclidean sample. 
Moreover, we prove the $\pi_1$-isomorphism of the natural shadow projection of the path-based Vietoris-Rips complexes of a sample lying in a Hausdorff proximity of a planar graph.
Since the homotopy reconstruction works in any host dimension $N\geq2$, the immediate next step is to consider three-dimensional graphs for geometric reconstruction.
 The difficulty lies in the elusive nature of the shadow beyond the plane; the shadow projection of (even) Euclidean Vietoris--Rips is not fully known to induce $\pi_1$-isomorphism in $\R^3$ \cite{adamaszek_homotopy_2017}. Nonetheless, the exploration remains very relevant to practical applications of Euclidean graph reconstruction.
The study sparks several interesting future research directions. 
Euclidean graphs are the simplest, albeit interesting, class of geodesic spaces one can consider. 
It is reasonable to believe that the geometric recovery of more general geodesic spaces---such as bounded curvature spaces considered by \cite{MajhiStability}---can similarly be approached using the shadow of homotopy equivalent Vietoris-Rips complexes. 

\subsection*{Acknowledgments}
Author SM would like to thank his middle school math teacher, Mr. Satyajit Das, for teaching him the intricacies of plane trigonometry using innovative methodologies and with utmost patience.

\begin{appendices}

\section{Additional Results and  Proofs}

\begin{proof}[Proof of Proposition~\ref{prop:radius}]
In the proof of Proposition~\ref{prop:center}, we saw that $c(A')$ is the (geodesic) midpoint of the two farthest points of $\mathcal A'$. 
Letting $r=\frac{1}{2}\diam_\G(\mathcal A)$, we note from Proposition~\ref{prop:center} that $\mathcal A'\subset \mathcal A\subset\overline{\B_\G(c(\mathcal A),r)}$. 
Also, the metric ball $\overline{\B_\G(c(\mathcal A),r)}$ is a geodesically convex subset of $\G$, due to the fact that $r<\ell(\G)/3$. 
So, we have that $c(\mathcal A')\in\overline{\B_\G(c(\mathcal A),r)}$. Hence, $d_\G(c(\mathcal A'),c(\mathcal A))\leq r=\frac{1}{2}\diam_\G(\mathcal A)$.
\end{proof}

\begin{proof}[Proof of Proposition~\ref{prop:delta_path}]
We consider the path $\gamma\colon[0,1]\to\R^N$ joining $a$ and $b$ by concatenating three pieces: the segment $[a,A]$ (of length $\leq\alpha$) with an $\eps$--path connecting $A$ and $B$ (of length $d^\eps_\s(A, B)$), and the segment $[B,b)]$ (of length $\leq \alpha$).

Since $\alpha<\eps/2$, it's not difficult to see that the image of $\gamma\subset \G^\eps$. 
This implies that $d_{\G^\eps}(a, b)\leq L(\gamma)\leq[d^\eps_\s(A,B)+2\alpha]$.
Since $d_\G(a,b)\geq R$, from the definition of restricted distortion, we now get
\[
d_\G(a,b)\leq\delta^\eps_R(\G)d_{\G^\eps}(a,b)\leq\delta^\eps_R(\G)\left[d^\eps_\s(A,B)+2\alpha\right].
\]
\end{proof}

\begin{proposition}[Path-connectedness]\label{prop:path-connected} 
Let $(\G,d_\G)$ be a path-connected graph embedded in $\R^N$ and $\s\subset\R^N$ with $d_H(\G,\s)<\alpha$. 
Then the geometric complex of $\Ri^\eps_{\beta}(\s)$ is
path-connected for any $\beta\geq\eps>2\alpha$.
\end{proposition}

\begin{proof}
Let us denote by $\Ri_\beta(\G)$ and $\Ri_\beta(\s)$ the Euclidean Vietoris--Rips complexes of $\G$ and $\s$, respectively.

Let $A,B\in \s$ be arbitrary points.
Then, there exist points $a,b\in\G$ such that $\max\{\norm{a-A},\norm{b-B}<\alpha$. 
Since $\G$ is assumed to be path-connected, so is $\Ri_{\eps-2\alpha}(\G)$. 
As a result, there exists a sequence
$\{x_i\}_{i=0}^{m+1}\subset\G$ forming a path in $\Ri_{\eps-2\alpha}(\G)$ joining $a$ and $b$. 
In other words, $x_0=a$, $x_{m+1}=b$, and $\norm{x_i-x_{i+1}}<\eps-2\alpha$ for $0\leq i\leq m$. 
There is also a corresponding sequence $\{X_i\}_{i=0}^{m+1}\subset\s$ such that $X_0=A$, $X_{m+1}=B$, and
$\norm{X_i-x_i}<\alpha$ for all $i$. 
We note that
\[
\norm{X_i-X_{i+1}}\leq\norm{x_i-x_{i+1}}+2\alpha<(\eps-2\alpha)+2\alpha=\eps.
\]
Thus, the sequence $\{X_i\}$ produces a path in $\Ri_{\beta+2\eps}(\s)$ joining $A$ and $B$. 
So, the geometric complex of $\Ri_{\beta+2\eps}(\s)$ is path-connected.
Since $\beta\geq\eps$, we have
the inclusion $\Ri_{\eps}(\s)\hookrightarrow\Ri^\eps_{\beta}(\s)$.
We conclude that $\Ri^\eps_\beta(\s)$ is path-connected.
\end{proof}

\begin{lemma}[Commuting Diagram \cite{majhi2023vietoris}]\label{lem:barycenter-map-lemma} Let $\K$ be a pure $m$--complex and $\L$ a flag complex. Let
$f\colon\K\to\L$ and $g\colon\sd\K\to\L$ be simplicial maps such that 
\begin{enumerate}
\item $g(v)=f(v)$ for every vertex $v$ of $\K$,
\item $f(\sigma)\cup g(\widehat{\sigma})$ is a simplex of $\L$ whenever
$\sigma$ is a simplex of $\K$.
\end{enumerate}
Then, the following diagram commutes up to contiguity: 
\begin{equation*}
\begin{tikzpicture} [baseline=(current  bounding  box.center)]
\node (k1) at (-2,-2) {$\sd{\K}$};
\node (k2) at (2,-2) {$\K$};
\node (k3) at (0,0) {$\L$};
\draw[map] (k2) to node[auto,swap] {${\sd}$} (k1);
\draw[map,swap] (k1) to node[auto,swap] {$g$} (k3);
\draw[map,swap] (k2) to node[auto] {$f$} (k3);
\end{tikzpicture}
\end{equation*}
where $\sd^{-1}$ is the linear homeomorphism sending each vertex of $\sd{\K}$ to the corresponding point of $\K$.
\end{lemma}

\begin{proof}[Proof of Proposition~\ref{prop:shadow-radius}]
Let $\delta$ be a small positive number such that 
\begin{align}\label{eq:delta}
\text{(i) }\delta\leq\Theta,
\text{ (ii) }(1+\delta)\leq\frac{1}{\sqrt{1-\Theta}}, \text{ (iii) }2\delta+\Theta\leq 1,
\text{ (iv) }(1+\delta)(\Theta+\delta)\leq\frac{1+\Theta}{2}.
\end{align}
Such $\delta$ can be chosen since $1/2\leq\Theta<1$.

Next, we pick sufficiently small $r>0$ such that: 
\begin{itemize}
\item $\|a-b\|\leq r$ implies $d_\G(a,b)<\ell(\G)/3$ so that there exists a unique geodesic joining them;
\item due to the assumption that each edge is $C^1$ (hence locally Lipschitz), both $a,b$ lying on the same edge of $\G$ with $\norm{a-b}\leq 2r$ implies 
\begin{equation}\label{eq:C^1-appx}
\text{(i) }d_\G(a,b)\leq (1+\delta)\overline{ab}~\text{ and (ii) }
ab^{(1-\Theta)d_\G(a,b)}\text{ contains the geodesic joining }a,b.
\end{equation} 
\item if $a,b\in\G$ with $\|a-b\|\leq r$ lie on two distinct edges $e_a,e_b$ of $\G$ incident to a vertex $v$, respectively, then
\begin{equation}\label{eq:C^1-angle}
\left|\cos^2{\frac{\angle{avb}}{2}}-\cos^2{\frac{\angle{e_ae_b}}{2}}\right|\leq\delta.
\end{equation} 
\item if $\|a-b\|\leq r$ and $q\in \G\cap ab^\eps$ for some $0\leq\eps\leq\frac{1}{2}[1-\Theta]^{3/2}r$, then one of the following must happen:
\begin{enumerate}[(A)]
\item at least one of the pairs $(a,q)$ or $(b,q)$ belong to the same edge of $\G$ as Figure~\ref{fig:shadow-rad-cases}(A);
\item $a,b$ belong to the same edge $e_a$ but $q$ belongs to a different edge $e_q$ Figure~\ref{fig:shadow-rad-cases}(B);
\item $a, b, q$ lie on three distinct edges of $\G$ incident to a common vertex $v\in\mathscr{V}(\G)$ as Figure~\ref{fig:shadow-rad-cases}(C). 
\end{enumerate}
\end{itemize}

\begin{figure}[hbt]
\centering
\begin{subfigure}[t]{0.3\textwidth}
\centering
\begin{tikzpicture}[scale=1]
\draw[thick, blue] (-2, 0) .. controls (-2,-1.5) .. node[auto,swap] {$e_a$} (0,-3) 
coordinate[pos=0.10] (q)
coordinate[pos=0.25] (a)
coordinate[pos=0.95] (b);    
\draw[dashed, gray] (a) -- (b);    
\fill[blue] (q) circle (2pt) node[blue, anchor=east] {$q$};
\fill[blue] (a) circle (2pt) node[blue, anchor=east] {$a$};
\fill[blue] (b) circle (2pt) node[blue, anchor=east] {$b$};
\end{tikzpicture}
\caption{Case A}
\end{subfigure}
~
\begin{subfigure}[t]{0.3\textwidth}
\centering
\begin{tikzpicture}[scale=1]
\draw[thick, blue, name path=1] (-2, 0) .. controls (-2,-1.5) .. node[auto,swap] {$e_a$} (0,-3) coordinate (d) node[anchor=west]  {$v$}
coordinate[pos=0.1] (a)
coordinate[pos=0.9] (b)
coordinate[pos=0.7] (q1);    
\draw[dashed, gray] (a) -- (b);
\draw[thick, blue] (0.5, 0.5)  .. controls (0,0) .. node[auto,swap] {$e_b$} (0,-3) coordinate[pos=0.7] (q) ; 
\fill (q1) circle (2pt);
\fill[blue] (a) circle (2pt) node[anchor=east] {$a$};
\fill[blue] (b) circle (2pt) node[anchor=east] {$b$};
\fill[blue] (q) circle (2pt) node[anchor=west] {$q$};
\draw[dashed,gray] (q) -- node[auto,black,swap] {$\eps$} (q1) node[black,anchor=north] {$p$};
\end{tikzpicture}
\caption{Case B}
\end{subfigure}
~
\begin{subfigure}[t]{0.3\textwidth}
\centering
\begin{tikzpicture}[scale=0.7]
\draw[thick, blue] (-2, 0) .. controls (-2,-1.5) .. node[auto,swap] {$e_a$} (0,-3) coordinate (d) node[anchor=west]  {$v$} .. controls (1,-1) .. node[auto,swap] {$e_b$} (3, 1);    
\draw[dashed, gray] (-2, 0) -- (0,-3) -- (3, 1);    
\draw[gray, dashed, name path=1] (-2, 0) coordinate (a) node[blue, anchor=east] {$a$} -- (3, 1) coordinate (b) node[blue, anchor=west] {$b$}; 
\draw[thick, blue] (1, 2) coordinate (c) node[anchor=south] {$q$} .. controls (0,0) ..  (0,-3); \draw[dashed, gray, name path=2] (0, -3) --  (1,2);   
\fill[blue] (0,-3) circle (2pt);
\fill[blue] (-2,0) circle (2pt);
\fill[blue] (1,2) circle (2pt);
\fill[blue] (3,1) circle (2pt);
\fill[fill=black, name intersections={of=1 and 2}] (intersection-1) circle (2pt) node[anchor=south west] {$p$};
\end{tikzpicture}
\caption{Case C}
\end{subfigure}
\caption{The Cases for Proposition~\ref{prop:shadow-radius}}
\label{fig:shadow-rad-cases}
\end{figure}

In the rest of the proof, we show that the property ($\star$) of Definition~\ref{def:shadow} holds in each of the above two cases.

\subsection*{Case (A)}
In this case, if $q$ lies on the geodesic joining $a$ and $b$, then we trivially get
\[\min\{d_\G(a, q), d_\G(b, q)\}\leq\frac{1+\Theta}{2}d_\G(a,b).\]
Otherwise, if $q$ lies outside the geodesic as in Figure~\ref{fig:shadow-rad-cases}(A), then there exists $q'\in ab$ with $\|q-q'\|<\eps$.
Due to~[\ref{eq:C^1-appx}(ii)] and the triangle inequality, we have $\max\{\norm{q-a}, \norm{q-b}\}\leq((1-\Theta)d_\G(a,b)+\eps)\leq(d_\G(a,b)/2+\eps)$ since $\frac{1}{2}\leq\Theta<1$.

Since $d_\G(a,b)/2+\eps\leq d_\G(a,b)/2+(1-\Theta)^{3/2}r/2\leq (1+\delta)r/2+r/2=(2+\delta)r/2\leq2r$, now [\ref{eq:C^1-appx}(i)] implies 
\begin{align*}
\min\{d_\G(a,q), d_\G(b,q)\}&\leq(1+\delta)(d_\G(a,b)/2+\eps)\\
&=\frac{1+\delta}{2}d_\G(a,b))+(1+\delta)\eps \\
&\leq\frac{1+\Theta}{2}d_\G(a,b)+\frac{\eps}{\sqrt{1-\Theta}},\text{ due to~[\ref{eq:delta}(i)] and due to~[\ref{eq:delta}(ii)], resp.}\\
&\leq\frac{1+\Theta}{2}d_\G(a,b)+\frac{\eps}{1-\Theta}.
\end{align*}

\subsection*{Case (B)} 
In this case, we assume $a,b\in e_a$ and $q\in e_b$.
Due to~[\ref{eq:C^1-appx}(ii)], there exists a point $p$ on the geodesic joining $a,b$ such that $\overline{pq}\leq (1-\Theta)d_\G(a,b)+\eps$.
We note that
\[
\frac{\overline{pv}}{\sin{\angle{pqv}}}
=\frac{\overline{qv}}{\sin{\angle{vpq}}}
=\frac{\overline{pq}}{\sin{\angle{pvq}}}.
\]
So, 
\begin{align*}
\overline{pv}+\overline{qv} &=\frac{\sin{\angle{pqv}}+\sin{\angle{vpq}}}{\sin{\angle{pvq}}}\overline{pq} \leq\frac{\sin{\angle{pqv}}+\sin{\angle{vpq}}}{\sin{\angle{pvq}}}((1-\Theta)d_\G(a,b)+\eps)\\
&=\frac{2\sin{\frac{\angle{pqv}+\angle{vpq}}{2}}\cos{\frac{\angle{pqv}-\angle{vpq}}{2}}}{2\sin{\frac{\angle{pvq}}{2}}\cos{\frac{\angle{pvq}}{2}}}((1-\Theta)d_\G(a,b)+\eps) \\
&=\frac{\cos{\frac{\angle{pvq}}{2}}\cos{\frac{\angle{pqv}-\angle{vpq}}{2}}}{\sin{\frac{\angle{pvq}}{2}}\cos{\frac{\angle{pvq}}{2}}}((1-\Theta)d_\G(a,b)+\eps)\\
&\leq\frac{(1-\Theta)d_\G(a,b)+\eps}{\sin{\frac{\angle{pvq}}{2}}}
=\frac{(1-\Theta)d_\G(a,b)+\eps}{\sqrt{1-\cos^2{\frac{\angle{pvq}}{2}}}}
\leq\frac{(1-\Theta)d_\G(a,b)+\eps}{\sqrt{1-\cos^2{\frac{\angle{e_ae_b}}{2}}-\delta}},\text{using \eqref{eq:C^1-angle}}\\
&\leq\frac{(1-\Theta)d_\G(a,b)+\eps}{\sqrt{1-\cos^2{\frac{\angle{e_ae_b}}{2}}}},\text{using \eqref{eq:C^1-angle}}\\
&\leq\frac{(1-\Theta)d_\G(a,b)+\eps}{\sqrt{1-\Theta}}.
\end{align*}
From~[\ref{eq:C^1-appx}(i)], we get using~[\ref{eq:delta}(ii)]
\[
d_\G(p,q)\leq\frac{(1+\delta)[(1-\Theta)d_\G(a,b)+\eps]}{\sqrt{1-\Theta}}
\leq\frac{(1+\delta)(1-\Theta)}{2\sqrt{1-\Theta}}d_\G(a,b)+\frac{(1+\delta)}{\sqrt{1-\Theta}}\eps
\leq\frac{1+\delta}{2}d_\G(a,b)+\frac{\eps}{1-\Theta}.
\]
Since $q$ lies on the geodesic joining $a$ and $b$, conclude \[
\min\{d_\G(a,q), d_\G(b,q)\}
\leq
\frac{1+\delta}{2}d_\G(a,b)+\frac{\eps}{1-\Theta}
<\frac{1+\Theta}{2}d_\G(a,b)+\frac{\eps}{1-\Theta}.
\]
The last inequality is due to~[\ref{eq:delta}(i)].

\subsection*{Case (C)}
In this case, we have, without any loss of generality, from Proposition~\ref{prop:optimization} that 
\begin{align}
\overline{av}+\overline{pv}
&\leq\cos^2{\left(\tfrac{\min\{\phi,\varphi\}}{2}\right)}(\overline{av}+\overline{bv})\nonumber\\
&\leq\left[\cos^2{\left(\tfrac{\min\{\angle{e_ae_b},\angle{e_be_c}\}}{2}\right)}+\delta\right](\overline{av}+\overline{bv}),~\text{ due to \eqref{eq:C^1-angle}}\nonumber\\
&\leq(\Theta+\delta)(\overline{av}+\overline{bv}),~\text{ from the definition of}~\Theta\nonumber\\
&\leq(\Theta+\delta)[d_\G(a,v)+d_\G(v,b)]
=(\Theta+\delta)d_\G(a,b),~\text{ since }d_\G(a,b)<\ell(\G)/3
\label{eq:estimate-1}.
\end{align}
On the other hand, since $q\in ab^\eps$, again from Proposition~\ref{prop:optimization} we get 
\begin{equation}\label{eq:estimate-2}
\overline{pq}
\leq\frac{\eps}{\sqrt{2(1-\cos^2{(\min\{\angle{e_ae_b}, \angle{e_b,e_c}\}/2)})}}
\leq\frac{\eps}{\sqrt{2(1-\Theta-\delta)}}
\leq\frac{\eps}{\sqrt{1-\Theta}}.    
\end{equation}
The last inequality is due to~[\ref{eq:delta}(iii)].
From~\eqref{eq:estimate-1} and~\eqref{eq:estimate-2}, we now get
\begin{align*}
\overline{vq}\leq(\overline{av}+\overline{vp})+\overline{pq}
\leq(\Theta+\delta)d_\G(a,b)+\tfrac{\eps}{\sqrt{1-\Theta}}\leq\left[(\Theta+\delta)+\tfrac{(1-\Theta)^{3/2}}{2\sqrt{1-\Theta}}\right]r
=\tfrac{1}{2}(2\delta+\Theta+1)r\leq r.
\end{align*}
Again, the last inequality is due to~[\ref{eq:delta}(iii)].

As a result, \eqref{eq:C^1-appx} implies $d_\G(v,q)\leq(1+\delta)\overline{qv}$.
Putting everything together, we finally get
\begin{align*}
d_\G(a, q) &= d_\G(a,v)+d_\G(v, q),\text{ since }d_\G(a,q)<\ell(\G)/3\\ 
&\leq (1+\delta)(\overline{av}+\overline{vq}),~\text{due to \eqref{eq:C^1-appx}}\\ 
&=(1+\delta)(\overline{av}+[\overline{vp}+\overline{pq}])
=(1+\delta)([\overline{av}+\overline{vp}]+\overline{pq})\\
&\leq(1+\delta)\left[(\Theta+\delta)d_\G(a,b)+\tfrac{\eps}{\sqrt{1-\Theta}}\right],~\text{ adding \eqref{eq:estimate-1} and \eqref{eq:estimate-2} }\\
&\leq(1+\delta)(\Theta+\delta)d_\G(a,b)+\tfrac{(1+\delta)}{\sqrt{1-\Theta}}\eps\\
&\leq\frac{1+\Theta}{2} d_\G(a,b)+\frac{\eps}{1-\Theta},\text{ from~[\ref{eq:delta}(iv)] and~[\ref{eq:delta}(ii)], respectively}.
\end{align*}
Hence the proof.
\end{proof}

\begin{proposition}\label{prop:optimization}
Let $\triangle avp$ and $\triangle bvp$ be two triangles as shown in Figure~\ref{fig:shadow-rad} with $\angle{avp}=\phi$, $\angle{bvp}=\varphi$, and $\overline{av}+\overline{vb}=r$. 
Then, 
\[
\min\left\{\overline{av}+\overline{pv}, \overline{bv}+\overline{pv}\right\}
\leq\cos^2{\left(\frac{\min\{\phi,\varphi\}}{2}\right)}r.
\]
Moreover, if $q\in ab^\eps$, then 
$\overline{pq}\leq\eps/\sqrt{2\left[1-\cos^2{\frac{\min\{\phi,\varphi\}}{2}}\right]}$.
\end{proposition}
\begin{proof}[Proof of Proposition~\ref{prop:optimization}]
Without any loss of generality, we assume that 
$\phi\geq\varphi$.
Let $\theta\coloneqq\angle{apv}$ and 
$a(\theta)\coloneqq\overline{av}$, $b(\theta)\coloneqq \overline{bv}$, $p(\theta)\coloneqq\overline{pv}$.
We immediately note from the triangles that $\theta\in(\varphi, \pi-\phi)$.

Under the constraints that (i) $\phi,\varphi$ are fixed angles and (ii) the sum $a(\theta)+b(\theta)=r$ is constant, we find the maximum value of $\min\left\{a(\theta)+p(\theta),b(\theta)+p(\theta)\right\}$ for $\theta\in(\varphi, \pi-\phi)$.

From the smaller triangles in Figure~\ref{fig:shadow-rad}, we have
\begin{equation}\label{eq:a-p-b}
\frac{p(\theta)}{\sin{(\theta+\phi)}}=\frac{a(\theta)}{\sin{\theta}}\text{ and }
\frac{p(\theta)}{\sin{(\theta-\varphi)}}=\frac{b(\theta)}{\sin{\theta}}.
\end{equation}

\begin{figure}[hbt]
\centering
\begin{tikzpicture}[scale=1.2]
\draw[name path=1] (-2, 0) coordinate (a) node[anchor=east] {$a$} -- (3, 1) coordinate (b) node[anchor=west] {$b$}; 
\draw (-2, 0) -- node[auto,swap] {$a(\theta)$} (0,-3) coordinate (v) node[anchor=west]  {$v$} -- node[auto,swap] {$b(\theta)$} (3, 1);    
\draw[name path=2] (1.2, 2.5) coordinate (p) node[anchor=south] {$q$} --  (0,-3);   
\draw[name intersections={of=1 and 2}] (intersection-1) coordinate (q) node[anchor=north west] {$p$} -- node[auto,swap] {$p(\theta)$} (0,-3);   
\pic [draw, -, "\footnotesize{$\phi$}", angle eccentricity = 1.5] {angle = c--v--a};
\pic [draw, -, "\footnotesize{$\varphi$}", angle eccentricity = 1.5] {angle = b--v--c};
\pic [draw, -, "\footnotesize{$\theta$}", angle eccentricity = 1.5] {angle = a--q--v};
\pic [draw, -, "\footnotesize{$\alpha$}", angle eccentricity = 1.5] {angle = b--q--p};
\draw[dashed] (p) -- node[auto] {$\eps$} ($(a)!(p)!(b)$);
\end{tikzpicture}
\caption{}
\label{fig:shadow-rad}
\end{figure}

Adding them, we get
\begin{align*}
& p(\theta)\left[ \frac{\sin{\theta}}{\sin{(\theta+\phi)}} + \frac{\sin{\theta}}{\sin{(\theta-\varphi)}} \right]=a(\theta)+b(\theta)\coloneqq r \\
\implies &p(\theta)=\frac{r}{\sin{\theta}}\bigg/\left[ \frac{1}{\sin{(\theta+\phi)}} + \frac{1}{\sin{(\theta-\varphi)}} \right].
\end{align*}
Note from \eqref{eq:a-p-b} that $a(\theta)=b(\theta)$ implies $\theta=\frac{\pi}{2}-\frac{\phi-\varphi}{2}$. 
We divide the domain $(\varphi,\pi-\phi)$ of $\theta$ into $I_a\coloneqq\left(\varphi,\frac{\pi}{2}-\frac{\phi-\varphi}{2}\right]$ and $I_b\coloneqq\left(\frac{\pi}{2}-\frac{\phi-\varphi}{2}, \pi-\phi\right)$.

\subsection*{($\pmb{I_a}$)}
For $\theta\in I_a$, we maximize 
\[x(\theta)\coloneqq p(\theta)+a(\theta)=r\left[ \frac{1}{\sin{\theta}} + \frac{1}{\sin{(\theta+\phi)}} \right]\bigg/\left[ \frac{1}{\sin{(\theta+\phi)}} + \frac{1}{\sin{(\theta-\varphi)}} \right].\]
We argue that $x(\theta)$ is an increasing function by showing the positivity of its derivative:
\begin{align*}
x'(\theta)=\tfrac{1}{2}\cos{\tfrac{\phi}{2}}\sin{\tfrac{\varphi}{2}}\sec{\tfrac{\phi+\varphi}{2}}\csc^2{\theta}
\csc^2{\tfrac{\phi-\varphi+2\theta}{2}}
\big[1+\cos{(\varphi-2\theta)-\cos{(\phi-\varphi+2\theta)}-\cos{(\phi+2\theta)}}\big].
\end{align*}
The last term (under brackets) is always positive. 
Indeed, its derivative is $2\sin{(\phi-\varphi+2\theta)}+2\sin{(\phi+2\theta)}+2\sin{(\varphi-2\theta)}>0$ and its value at $\theta=0$ is $4\sin{\frac{\phi}{2}}\cos{\frac{\varphi}{2}}\sin{\frac{\phi-\varphi}{2}}>0$ since $\varphi/2<\pi/2$. 
So, we conclude that $x(\theta)$ is an increasing function of $\theta$.

\subsection*{($\pmb{I_b}$)}
For $\theta\in I_b$, we maximize 
\[y(\theta)\coloneqq p(\theta)+b(\theta)=r\left[ \frac{1}{\sin{\theta}} + \frac{1}{\sin{(\theta-\varphi)}} \right]\bigg/\left[ \frac{1}{\sin{(\theta+\phi)}} + \frac{1}{\sin{(\theta-\varphi)}} \right].\]
We argue that $y(\theta)$ is a decreasing function by showing the negativity of its derivative:
\begin{align*}
y'(\theta)=\tfrac{1}{2}\cos{\tfrac{\phi}{2}}\sin{\tfrac{\varphi}{2}}\sec{\tfrac{\phi+\varphi}{2}}\csc^2{\theta}
\csc^2{\tfrac{\phi-\varphi+2\theta}{2}}
\big[\cos{(\phi-\varphi+2\theta)}-\cos{(\phi+2\theta)}+\cos{(\varphi-2\theta)}-1\big].
\end{align*}
The last term (under brackets) is always negative. 
Indeed, $\theta\geq\frac{\pi}{2}-\frac{\phi-\varphi}{2}$ implies that $\cos{(\phi-\varphi+2\theta)}\leq\cos{(\phi+2\theta)}$ since the cosine function is increasing in the interval $[\pi, 2\pi]$.

As a consequence, for any $\theta\in(\varphi, \pi-\phi)$:
\begin{align*}
\min\{x(\theta), y(\theta)\} &\leq
x\left(\tfrac{\pi}{2}-\frac{\phi-\varphi}{2}\right)
=y\left(\tfrac{\pi}{2}-\frac{\phi-\varphi}{2}\right) \\
&=r\left[ \frac{1}{\cos{\frac{\phi-\varphi}{2}}} + \frac{1}{\cos{\frac{\phi+\varphi}{2}}} \right]\bigg/\frac{2}{\cos{\frac{\phi+\varphi}{2}}}\\
&=\frac{\cos{\frac{\phi+\varphi}{2}}+\cos{\frac{\phi-\varphi}{2}}}{2\cos{\frac{\phi-\varphi}{2}}}r
=\frac{2\cos{\frac{\phi}{2}}\cos{\frac{\varphi}{2}}}{2\cos{\frac{\phi-\varphi}{2}}}r\\
&=\frac{\cos{\frac{\phi}{2}}\cos{\frac{\varphi}{2}}}{\cos{\frac{\phi}{2}}\cos{\frac{\varphi}{2}}+\sin{\frac{\phi}{2}}\sin{\frac{\varphi}{2}}}r\\
&=\frac{1}{1+\tan{\frac{\phi}{2}}\tan{\frac{\varphi}{2}}}r\leq\frac{1}{1+\tan^2{\frac{\varphi}{2}}}r,\text{ since }\phi\geq\varphi\\
&=\cos^2\frac{\varphi}{2}r.
\end{align*}

Finally, if $q\in ab^\eps$, then the top-most triangle implies that $\overline{pq}\leq\csc{(\alpha)}\eps$, where is $\alpha<\pi/2$ is either $\theta$ or $(\pi-\theta)$.
Since $\theta\in(\varphi, \pi-\phi)$, we then have $\csc{\alpha}\leq\csc{\varphi}$. Therefore, $\overline{pq}\leq\csc{(\varphi)}\eps=\eps/\sqrt{2[1-\cos^2{(\varphi/2)}]}$.

This completes the proof.
\end{proof}

\begin{proof}[Proof of Proposition~\ref{prop:intersection}]
Let $\mathcal{P}=\{P_i\}_{i=0}^{n+1}\in\mathscr{P}^\eps_\s(A,B)$ and $\mathcal{Q}=\{Q_j\}_{j=0}^{m+1}\in\mathscr{P}^\eps_\s(C,D)$ be the shortest $\eps$-paths, i.e., $d^\eps_\s(A,B)=L(\mathcal P)$ and $d^\eps_\s(C,D)=L(\mathcal Q)$. We now consider the following two cases, depending on whether the sequences of line segments of $\mathcal P$ and $\mathcal Q$ intersect as subsets of $\mathbb{R}^2$. \\

\noindent\textbf{Case A.}
Let us assume that (images of) the $\eps$--paths intersect at a point $I\in\R^2$ in this case.
Consequently, there exist $0\leq i\leq n$ and $0\leq j\leq m$ such that $I$ is the point of intersection of the segments $P_iP_{i+1}$ and $Q_jQ_{j+1}$, creating four (possibly degenerate) half-segments: $P_iI, P_{i+1}I, Q_jI, Q_{j+1}I$. 

Without any loss of generality, we assume that $P_iI$ is the shortest of the four.
As a result, we note from the triangle inequality that 
\[\overline{P_iQ_j}\leq \overline{P_iI}+ \overline{Q_jI}\leq \overline{Q_{j+1}I}+ \overline{Q_jI}=\overline{Q_iQ_{j+1}}<\eps.\] 
Similarly, $\overline{P_iQ_{j+1}}<\eps$.

Set $E\coloneqq P_i$ to immediately note that $d^\eps_\s(A,B,E)= d^\eps_\s(A,B)$.
Moreover, the above implies that $\mathcal Q'\coloneqq\{C=Q_0, Q_1, \ldots, Q_{j}, P_i=E\}$ and $\mathcal Q''\coloneqq\{E=P_i, Q_{j+1}, \ldots, Q_{m}, Q_{m+1}=D\}$ are $\eps$--paths with their length at most $L(\mathcal Q)$, implying $\diam^\eps_\s(C, D, E)= d^\eps_\s(C,D)$. 

Therefore, we get $\max\{\diam^\eps_
S\{A, B, E\},\diam^\eps_\s\{C, D, E\}\}\leq\max\{d^\eps_\s(A,B),d^\eps_\s(C,D)\}$. \\

\noindent\textbf{Case B.}
In this case, we assume the $\eps$--paths don't interact.
The Jordan curve theorem \cite{Moise1977} implies, without any loss of generality, that the image of $AB$ transversely intersects $\mathcal Q$ as subsets of the plane.
Consequently, there exists $0\leq j\leq m$ such that $Q_iQ_{j+1}$ transversely intersects $AB$.
Since $\norm{Q_j-Q_{j+1}}<\eps$, we assume, without any loss of generality, that $Q_j\in AB^{\eps/2}$.
Let $q_j,q_{j+1}\in \G$ be such that $\max\{|Q_j-q_j|,|Q_{j+1}-q_{j+1}|\}<\frac{1}{2}\xi\eps$.
The triangle inequality implies that $q_j\in ab^{(1+\xi)\eps/2}$.

By the definition of $\Delta(\G)$ and the fact that $(1+\xi)\eps/2\leq(1-\Theta)^{3/2}\Delta(\G)/2$, we have
\[
\min\{d_\G(a,q_j),d_\G(b,q_{j})\}
\leq\frac{1+\Theta}{2}d_\G(a,b)+\frac{(1+\xi)\eps/2}{1-\Theta}.
\]
So, Proposition~\ref{prop:d^esp-d^L-estimate} implies that 
\[
\min\{d^\eps_\s(A,Q_j),d^\eps_\s(B,Q_{j})\}
\leq\frac{\min\{d_\G(a,q_j),d_\G(b,q_{j})\}+\xi\eps}{1-\xi}
\leq\frac{(1-\Theta^2) d_\G(a,b)+(1+3\xi-2\xi\Theta)\eps}{2(1-\xi)(1-\Theta)}.
\]
Since $\norm{Q_j-Q_{j+1}}<\eps$, we consequently have 
\begin{align*}
\min\{d^\eps_\s(A,Q_{j+1}),d^\eps_\s(B,Q_{j+1})\}
&\leq\frac{(1-\Theta^2) d_\G(a,b)+(1+3\xi-2\xi\Theta)\eps}{2(1-\xi)(1-\Theta)}+\eps\\
&\leq\frac{(1-\Theta^2) d_\G(a,b)+(3+\xi-2\Theta)\eps}{2(1-\xi)(1-\Theta)}.
\end{align*}
Set $E\coloneqq Q_j$ and $F\coloneqq Q_{j+1}$ to immediately note that $\diam^\eps_\s(C,D,E, F)=d^\eps_\s(C,D)$.
Our argument above implies 
\[
\min\{\diam^\eps_\s(A, E, F),\diam^\eps_\s(B, E, F)\}\leq\frac{(1-\Theta^2) d_\G(a,b)+(3+\xi-2\Theta)\eps}{2(1-\xi)(1-\Theta)}.
\]
This completes the proof.
\end{proof}

\end{appendices}

\bibliographystyle{plain}
\bibliography{main}

\end{document}